\documentclass[final,5p,times,twocolumn]{elsarticle}

\usepackage{amsmath}
\usepackage{mathtools}
\usepackage{amssymb}

\usepackage{amsthm}
\usepackage{amsfonts}
\usepackage{graphicx}
\usepackage[hyphens]{url}
\usepackage[hidelinks]{hyperref}
\hypersetup{breaklinks=true}
\usepackage{nicefrac}
\usepackage{mathrsfs} 
\usepackage{upgreek}
\usepackage{siunitx}
\usepackage{algorithm}
\usepackage{array}
\usepackage{multirow}
\usepackage{makecell}
\usepackage{algpseudocode}
\usepackage{dsfont}
\usepackage{amsmath}
\usepackage{todonotes}
\usepackage{booktabs,caption}
\usepackage{color,soul}
\usepackage{subfigure}
\usepackage[labelfont=bf]{caption}
\usepackage{natbib}

\newtheorem{defn}{Definition}
\newtheorem{thm}{Theorem}
\newtheorem{lem}{Lemma}
\newtheorem{prop}{Proposition}
\newtheorem{cor}{Corollary}

\newtheorem{rem}{Remark}
\newtheorem{exam}{Example}
\newtheorem{assump}{Assumption}

\newcommand{\R}{\mathbb{R}}
\newcommand{\N}{\mathbb{N}}
\newcommand{\X}{\mathcal{X}}
\newcommand{\G}{\mathcal{G}}
\newcommand{\E}{\mathcal{E}}
\newcommand{\V}{\mathcal{V}}
\newcommand{\U}{\mathcal{U}}
\newcommand{\W}{\mathcal{W}}
\newcommand{\p}{\mathrm{p}}
\newcommand{\h}{\mathrm{h}}
\newcommand{\pe}{\mathrm{pe}}
\newcommand{\ext}{\mathrm{ext}}
\newcommand{\LL}{\mathcal{L}}

\journal{Applied Energy}

\begin{document}

\begin{frontmatter}
\title{Economic Dispatch of a Single Micro Gas Turbine Under CHP Operation with Uncertain Demands}
\author[KTH]{Miel~Sharf\corref{CorAut}} \ead{sharf@kth.se}
\author[IIT]{Iliya~Romm} \ead{iliya@technion.ac.il}
\author[IIT]{Michael~Palman} \ead{p.michael@technion.ac.il}
\author[IIT]{Daniel~Zelazo} \ead{dzelazo@technion.ac.il}
\author[IIT]{Beni~Cukurel} \ead{beni@cukurel.org}

\address[KTH]{Division of Decision and Control Systems, KTH Royal Institute of Technology, and Digital Futures. 10044 Stockholm, Sweden.}

\address[IIT]{Department of Aerospace Engineering, Technion - Israel Institute of Technology, Haifa, Israel.}

\cortext[CorAut]{Corresponding Author.}

\begin{abstract}
This work considers the economic dispatch problem for a single micro gas turbine, governed by a discrete state-space model, under combined heat and power (CHP) operation and coupled with a utility. If the exact power and heat demands are given, existing algorithms can be used to give a quick optimal solution to the economic dispatch problem. However, in practice, the power and heat demands can not be known deterministically, but are rather predicted, resulting in an estimate and a bound on the estimation error. We consider the case in which the power and heat demands are unknown, and present a robust optimization-based approach for scheduling the turbine's heat and power generation, in which the demand is assumed to be inside an uncertainty set. We consider two different choices of the uncertainty set relying on the $\ell^\infty$- and the $\ell^1$-norms, each with different advantages, and consider the associated robust economic dispatch problems. We recast these as robust shortest-path problems on appropriately defined graphs. For the first choice, we provide an exact linear-time algorithm for the solution of the robust shortest-path problem, and for the second, we provide an exact quadratic-time algorithm and an approximate linear-time algorithm. The efficiency and usefulness of the algorithms are demonstrated using a detailed case study that employs real data on energy demand profiles and electricity tariffs.
\end{abstract}


\end{frontmatter}

\section{Introduction}
In recent years, combined cycle systems, in which local consumers provide electricity, hot water and heat for themselves, have become popular \cite{Pilavachi2002}. The attractiveness of such combined heating and power (CHP) units was shown in recent studies \cite{Gu2014,Liu2014}, and the economically favorable conditions toward integrating micro gas turbines (MGT) powered CHP units into the smart-grid was examined in \cite{Mongibello2015} and \cite{Pantaleo2013}. However, these works consider a generic MGT model and do not include realistic demand profiles nor the variable pricing of electricity. More recently, \cite{Rist2017} presented a solution to the CHP economic dispatch (ED) problem for a single MGT coupled to the utility with a realistic MGT performance model and known demand, i.e., an economically-optimal schedule of the MGT was computed for a consumer generating its own power and heat. In this paper, we propose a solution to a similar economic dispatch problem for the case of unknown demand by using the framework of robust optimization.

\subsection{Micro Gas Turbines}\label{subsec.mgt}
Micro gas- turbines (MGT) offer many advantages for small-scale CHP production, such as low greenhouse gas emissions, theoretical high thermal efficiency and reduced noise. They are also capable of short start-up times and rapid transitions between partial and full-load, due to their low mechanical and thermal inertia. For these reasons, \cite{Mongibello2015} and \cite{Pantaleo2013} examined the economically favorable conditions of integrating MGTs into the smart-grid.
However, theoretical analysis of MGTs, especially in an economic framework, can be hard due to physical limitations. These include, but are not limited to:
\begin{itemize}
    \item[i)] Many MGTs can only shutdown from or startup to certain operation levels \cite{C65Capstone,AMTNikeManual}.
    \item[ii)] When some MGTs are turned off, they must be cooled down before they can be turned on again. For example, the Capstone C65 MGT must cool down for up to 10 minutes before coming online again \cite{C65Capstone}.
    \item[iii)] Even when the turbines are active, not any generation level between the maximum and minimum capacity is allowed. This is due to structural and rotordynamic resonances rendering the engines unstable or unsafe for certain rotation speeds \cite{Rist2017,Lee1993,Papageorgiou2007,FederalRegister}.
    \item[iv)] From an aerodynamic perspective, compressor blade fluttering introduces additional permeating operational boundaries, which may interrupt the continuity of engine's operating line \cite{CapeceV,Bendiksen1988}.
    \item[v)] Gas turbine emissions such as carbon monoxide (CO), unburned hydrocarbons (UHC) and nitric oxides (NOx) cause an increasing concern. Percentage of CO, HC, NOx is directly correlated to the combustor temperature, equivalent ratio and pressure \cite{LieuwenT}, which are highly variant throughout engine's operating region. Towards reducing the amount of emissions, authorities impose strict regulations on gas turbine operators, which create zones in the operating line which are undesirable. Furthermore, the majority of modern engines with reduced emissions are operating with lean combustion, which is more prone to exhibit thermoacoustic instability (interaction between an acoustic field and a combustion process that increases pressure oscillations that may even lead to complete failure of the gas turbine unit). \cite{KellerJ,TimothyC}. Then, avoiding the combustor thermoacoustic instabilities also impose additional discontinuities in engine's operational field.
\end{itemize}
Any thorough economic analysis of a system including MGTs, including economic viability of MGTs or optimal generation planning, must account for the physical limitations of the MGT.

\subsection{Economic Analysis of Power Generation}
The economic analysis of power generation is usually done by considering the Economic Dispatch (ED) optimization problem. Generally, ED considers a collection of supply mechanisms generating power and/or heat, where the goal is to schedule the machines' generation to guarantee that the demand is met, while minimizing the overall production costs \cite{Happ1977}. This paper considers the ED problem for a single MGT and a utility, from which both power and heat can be purchased\footnote{\label{note.heat_buy}Heat is not directly sold by the utility, but can be modeled as an additional fuel or electricity cost. Most consumers satisfy their heat demand with a boiler, in which case we model the heating cost with the price of natural gas.}. Thus, the ED problem must account for the physical limitations i)-iii), as well as other physical limitations that the particular MGT model might possess.

Economic dispatch has been considered for many different types of systems, including steam engines, gas turbines, and wind turbines \cite{Rist2017,Wood2013,Bertsekas1983}. Most literature on ED simulate the generators as having continuous states based on first principle modeling of the system, where the generated power and heat can take any value between a minimum and a maximum capacity, and the corresponding cost function, mapping generation level to economic cost, is assumed to be quadratic \cite{Wood2013,Bertsekas1983}. The resulting ED problem is usually solved using standard convex optimization techniques, e.g. gradient descent or dual-gradient methods \cite{Bertsekas1983,Zelazo2012}. These methods can also be combined with other techniques, e.g. consensus-based algorithms \cite{Binetti2014}. \textcolor{black}{More recent works try to apply learning-based approaches \cite{Kim2020,Zhou2020} or particle swarm optimization methods \cite{Gaing2003,Xin2020} to solve the convex optimization problem. The reader is referred to the following recent reviews on the subject for more information and references \cite{Kazda2020,Wen2021}.}

Unfortunately, this convex optimization framework fails to capture the fundamental constraints imposed by the physical limitations of MGTs, e.g. points i)-v) described in Section \ref{subsec.mgt}, unless augmented properly, for multiple reasons. First, in a low-demand scenario, not all providers should be active, so we should also schedule their startup and shutdown. This is usually done by considering the unit commitment (UC) problem \cite{Saravanan2013}. However, due to the flexibility of MGTs and their quick start up and shutdown times, as well as the physical limitation i), this decoupling will result in a wasteful scheduling policy. Therefore, we do not decouple the UC and ED problems, and instead incorporate the inactive state into the ED problem. This is usually done by introducing binary variables determining when the machines should be active \cite{Wood2013,Bertsekas1983}. Moreever, the physical limitation ii) implies we need more than one inactive state per MGT. Furthermore, the physical limitations iii-v) mean that we cannot model the generation of the turbine (when active) as a continuous variable with minimum and maximum capacities. Other turbine-specific limitations can impose additional constraints on the model.

\subsection{Shortest-Path Algorithms and Uncertain Demand}
Combining the restrictions described above, we get a model for the MGT having multiple discrete (inactive) states and complex constraints on the allowed generation level when active, meaning that the ED problem is a constrained mixed-integer problem, which can be NP-hard in general. One possible solution is to discretize the state space and cost function, which works well for complex engine models, as the fuel consumption can be computed numerically. In this setting, the combined UC and ED problem is a discrete-time optimal control problem with a discrete state-space representation for the plant. This is an integer optimization problem where generated power and heat can only take values within a finite set. A solution to this problem is available using dynamic programming, namely by using the shortest-path algorithm on an appropriately defined graph \cite{Rist2017,Cormen2009}. However, this method, as well as most other approaches for ED, assumes the demands are known throughout the time horizon \cite{Rist2017,Bertsekas1983,Zelazo2012,Hindi1991,Ross1980,Kanchev2014,Shamsi2016}, e.g. by using one of the many load forecasting techniques that appear in the literature, see e.g. \cite{Gross1987,Akay2007,Yu2017,Mirasgedis2006} and references therein.

One might try to simply ignore the issue of unknown future demands by solving the ED problem with respect to an ad-hoc estimate of the demand level. However, this approach can fail miserably, as is known that for some real-world optimization problems, the optimal solution changes drastically when some parameters in the problem change even by a minuscule amount \cite{BenTal2009,Elsayed2016}. Another approach to overcome this problem is to consider a stochastic optimization framework, in which we try and minimize the average cost of generation \cite{Zhao2010,Hetzer2008,Dhillon1993}. These require prior knowledge on the probability distribution of the underlying uncertainty, which must be estimated from past data, resulting again in the same problem of parameter inaccuracy. Another approach taken by recent studies is the incorporation of robust optimization techniques, in which the demand is assumed to be in a given set, which is known as the ``uncertainty set". \textcolor{black}{The choice of uncertainty set requires us to have knowledge about the possible values the demand can take, which again results in a problem of parameter inaccuracy. Fortunately, there is evidence that robust algorithms for general problems are significantly less vulnerable to parameter uncertainty \cite{Bertsimas2004}.}
However, robust shortest-path problems are known to be generally NP-hard \cite{Yu1998}, meaning that a careful treatment of the problem and the uncertainty set is needed to assure that the resulting optimization problem is tractable. We discuss in detail about previous results regarding robust shortest-path problems in Section \ref{sec.Background_RSPP} below, but all existing solution methods are either overconservative or suffer from very prolonged runtimes even for small graphs with a few hundred nodes. For comparison, the MGT ED problem with known demand in \cite{Rist2017} is converted to a shortest-path problem on a graph with roughly $250,000$ nodes.

\subsection{Contributions}
In this work, we consider the ED problem of a single MGT with a known discrete state-space representation, and unknown power and heat demands. The turbine is also connected to a utility from which power and heat can be purchased at a time-dependent  cost\footnotemark[\value{footnote}]. We apply the robust optimization framework for the mixed-integer ED problem, which results in a robust shortest-path problem. We study multiple possible choices for the uncertainty set. In the first case, the demand at each time is within a given confidence interval. In the second case, we similarly restrict the demand at each time to lie inside confidence intervals with given centers and radii, but a certain bound (a "budget") is put on the aggregate deviation of the demand from the interval centres throughout the time horizon, i.e. the uncertainty is "budgeted" throughout the time horizon. In both cases, we present linear-time algorithms for finding the optimal solution, and prove their validity. To the best of the authors' knowledge, the algorithms we present are the first to give a tractable solution to the robust shortest path problem when the edge costs are positively correlated (see Section \ref{sec.Background_RSPP} for more details).

The paper is structured as follows. Section \ref{sec.Background} presents some background about ED, the shortest-path problem and robust optimization, as well as a literature review on the robust shortest path problem. Section \ref{sec.RobustDispatch} considers the robust ED problem as a worst-case shortest-path problem, including multiple possible cases for the uncertainty set, and presents efficient algorithms for solving the robust ED problem in these cases. Section \ref{sec.CaseStudy} portrays a case study demonstrating the algorithms. 

\paragraph*{Preliminaries}
We let $\N = \{1,2,\ldots\}$ be the set of all natural numbers. We use notions from graph theory \cite{Bondy1977}. A directed graph $\G = (\V,\E)$ consists of a finite set of vertices $\V$ and a set of edges $\E$ which are pairs of vertices of $\V$. An edge from $u\in \V$ to $v\in \V$ will be denoted as $u\to v$, where $u$ is the tail of the edge and $v$ is its head. A path from a vertex $u$ to a vertex $v$ is a sequence of edges $e_1,\cdots,e_\ell$ such that $u$ is $e_1$'s tail, $v$ is $e_\ell$'s head, and for any $i$, $e_i$'s head is $e_{i+1}$'s tail. A directed graph $\G$ is called a DAG (directed acyclic graph) if there are no paths which begin and end at the same vertex. For a node $v\in \V$, the in-degree $\deg(v)$ is the number of edges $e\in \E$ which have $v$ as a head. As each edge has exactly one head, we have that $\sum_{v\in \V} \deg(v) = |\E|$.
A weighted directed graph is a triplet $(\V,\E,w)$ where $(\V,\E)$ is a directed graph and $w:\E\to \mathbb{R}$ is called the weight function. The cost of a path is defined as the sum of the weights of its edges. The shortest path problem for a graph $\G$ is a combinatorial optimization problem in which the goal is to find the path with the smallest cost from a node $s$ to a node $q$. 

We consider some notions from convex analysis \cite{Rockafellar1970}. For a convex set $\W\subset \R^d$, we say $x\in \W$ is an \emph{extreme point} if for any $y,z\in \W$ and any $t\in(0,1)$, if $x = ty+(1-t)z$ then $x=y=z$. The collection of extreme points of $\W$ is denoted by $\ext(\W)$. If $f:\W \to \mathbb{R}$ is a convex function and $\W$ is bounded and closed, it is known that $\max_{x\in \W} f(x) = \max_{x\in \ext(\W)} f(x)$ \cite[Theorem 32.2]{Rockafellar1970}. For a norm $\|\cdot\|$ on $\R^N$, the norm ball of radius $r>0$ around $x_0\in \R^N$ is equal to $\{x\in \R^N: \|x-x_0\|\le r\}$. Moreover, a weighted $\ell^\infty$ norm on $\R^N$ is given by $\|x\| = \max_{i=1}^N \{w_i|x_i|\}$ where $w_1,\ldots,w_N>0$ are the associated weights. Similarly, a weighted $\ell^1$-norm is given by $\|x\| = \sum_{i=1}^N w_i|x_i|$. 
The Minkowski sum of two sets $A,B$ is given by $A+B = \{a+b: a\in A, b\in B\}$.

\section{\textcolor{black}{Turbine Models, ED, and Robust Optimization}} \label{sec.Background}
This section provides the required background material, including a model for the MGT, the ED and shortest-path problems, and some basic notions from robust optimization.
\subsection{Discrete State-Space Models for Turbines}
We consider a micro gas turbine (MGT) with a discrete state space, which is a generalization of \cite{Rist2017}. Due to the low mechanical and thermal inertia of the proposed class engine and with sufficiently high discretization resolution, transition time between engine states becomes negligible. Therefore, only MGT steady states are considered. We denote the state space of the MGT by $\X$, which is assumed to be a finite set. The state $x(t)$ of the turbine  evolves in discrete time. The dynamics can be modeled by two functions $f,c$ and a state-indexed set $\U(x)$, i.e., for each $x\in \X$, we denote the set of admissible control signals by $\U(x)$. The function $f(x,u)$ describes the allowable transitions between turbine states, and the function $c(x,u)$ describes the transition times between states. More precisely, we assume that the function $c(x,u) \in \N$ for all pairs $(x,u)$, and consider the equation governing the state evolution of the turbine:
\begin{align*}
x(t+c(x(t),u(t))\Delta t) = f(x(t),u(t)),
\end{align*}
where $u(t)$ is the control input at time $t$ and $\Delta t$ is the time increment. In other words, if the control input $u(t)$ is applied at the state $x(t)$, the next state will be $f(x(t),u(t))$ and it will take $c(x(t),u(t))\Delta t$ time to get there. As $c(x,u) \in \N$ for all pairs $(x,u)$, the state $x(t)$ evolves at time $0,\Delta t,2\Delta t,\ldots$.

For each state $x$ and control input $u$, we let $P_{\rm MGT}(x,u)$ be the power generation associated with the state-control pair $(x,u)$ for one time step $\Delta t$, and let $H_{\rm MGT}(x,u)$ be the heat generation associated with the same state-control pair for one time step $\Delta t$. For example, let $x\in \X$ be a state in which the turbine is switched off, and $u$ is a control input for which $f(x,u) = x$, i.e. the turbine is also switched off at the next time step, then $P_{\rm MGT}(x,u) = H_{\rm MGT}(x,u) = 0$. We emphasize that both the power and heat generation can also depend on $u$. Indeed, because $P_{\rm MGT}$ and $H_{\rm MGT}$ aggregate the generation between two times $t_0$ and $t_0+\Delta t$, the control $u(t_0)$ does not only determine the state of the turbine at time $t_0 + \Delta t$, but also the generated amount in the intermediate time. Indeed, the turbine generates power in continuous-time, even though our model is discrete-time.

\begin{rem}
The discretization methodology presented is general for any micro gas turbine, and does not inherently assume a specific physical model while solving the ED problems in Section \ref{sec.RobustDispatch}. 
More precisely, given any data set predicting the performance of the turbine analytically, numerically or empirically, we can achieve a discrete model by choosing a grid for each of the turbine parameters that define its state. Then, we compute the power output $P_{\rm MGT}(x,u)$, the heat output $H_{\rm MGT}(x,u)$, and the fuel flow consumed by the turbine for each possible transition between two states. The fuel flow consumption defines the cost of operating the turbine, and is described in Section \ref{sec.ED}.
Moreover, we can choose to impose certain generation or ramp-rate constraints by limiting either the allowable physical state, the allowable control actions, or pairs thereof before performing the discretization.
\end{rem}
\begin{exam}
Consider the MGT model in \cite{Rist2017}, consisting of a single stage centrifugal compressor, a can-type combustor, a single stage turbine, a recuperator and a separate heat recovery unit. There, the active states of the turbine are characterized by two parameters, $p$ and $h$. The variable $p$ is the speed of the engine, and can take values $\p_1,\cdots\p_{\rm s}$, while the variable $h$ is the position of a recuperator bypass valve, and can take values $\h_1,\cdots,\h_{\rm v}$, so $\X = \{(\p_i,\h_j) : i=1,\ldots,{\rm s},~j=1,\ldots,{\rm v}\}$. The allowable transitions change the speed of the engine, the recuperator bypass valve position, or both by one level. Changing the position of the valve or slowing down the engine takes one unit of time (i.e. $c(x,u) = 1$ in this case), while revving up the engine takes two units of time (so $c(x,u) = 2$ for this transition).
\end{exam}

\begin{figure} [!b] 
	\vspace{-15pt}
    \centering
    \includegraphics[width = 0.42\textwidth]{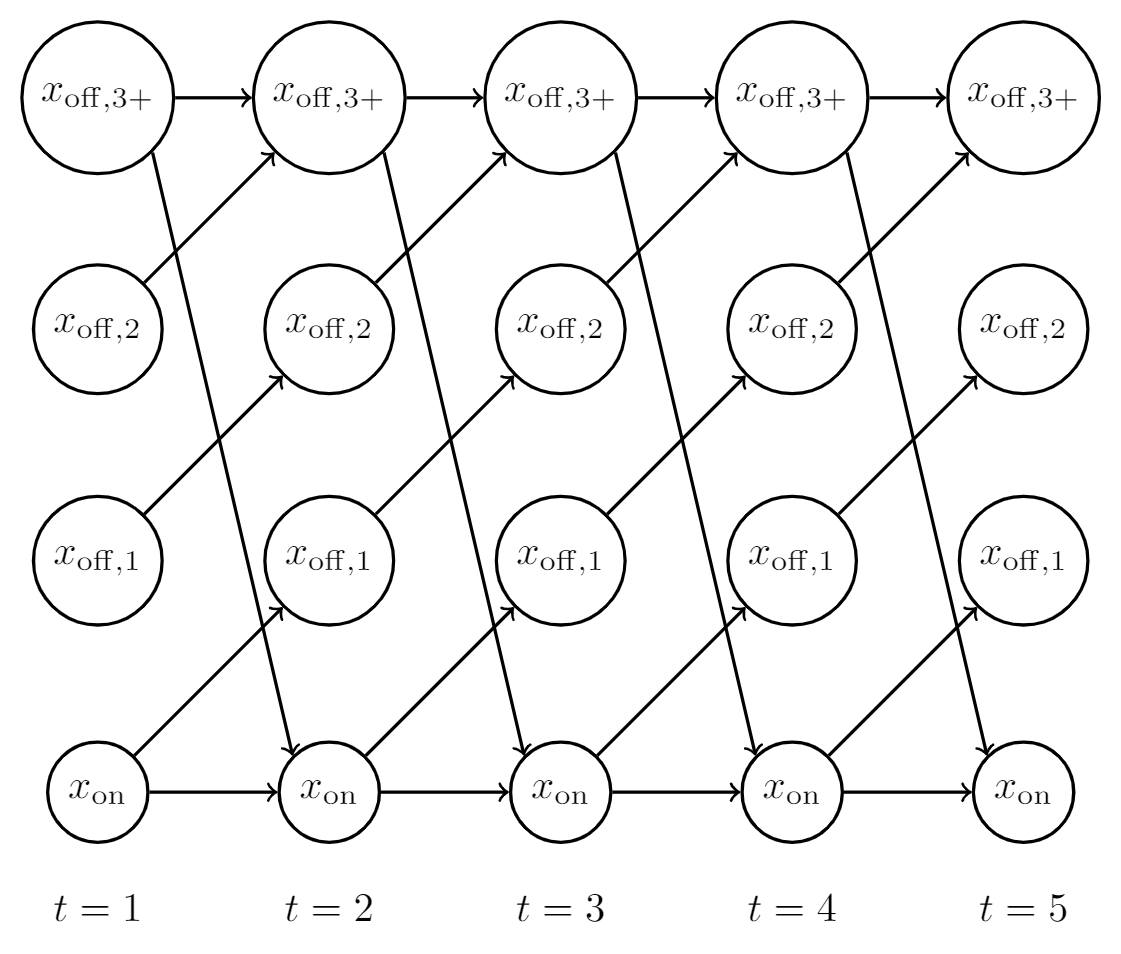}
	\vspace{-5pt}
    \caption{The state transition graph corresponding to the ED problem for the turbine in Example \ref{eq.StartUpMinimum} with time horizon $T=5$.}
    \label{fig.StartUpMinimum}
\end{figure}

\begin{exam} \label{eq.StartUpMinimum}
Consider a turbine that generates $\p$ units of power and $\h$ units of heat when active. When the turbine is on, it can be turned off at any time, within $15$ seconds. However, once it is turned off, it has a cool-down time of $45$ seconds (i.e., three time steps). We model the turbine using a discrete state-space representation with time instances $\Delta t =15_{\rm sec}$ apart and with $|\X| = 4$ possible states - one active state, $x_{\rm on}$, and three off states, $x_{\rm off,1}, x_{\rm off, 2}, x_{\rm off, 3+}$, which represent that the turbine has been inactive for 1,2, or at least 3 units of time, respectively. Here, $\U(x_{\rm on}) = \{\rm keep, shutdown\}$, $\U(x_{\rm off,3+}) = \{\rm keep, start\}$, and $\U(x_{\rm off,1}) = \U(x_{\rm off,2}) = \{\rm keep\}$. The control signal "$\rm keep$" moves $x_{\rm on}$ to itself, $x_{\rm off,1}$ to $x_{\rm off,2}$, $x_{\rm off,2}$ to $x_{\rm off,3+}$ and $x_{\rm off,3+}$ to itself. Moreover, the control signal ``$\rm shutdown$" moves $x_{\rm on}$ to $x_{\rm off,1}$, and the control signal ``$\rm start$" moves $x_{\rm off,3+}$ to $x_{\rm on}$. These transitions all take one time step, i.e. $c(x,u) = 1$ for all pairs $(x,u)$. The possible evolution of the state $x(t)$ of the turbine across 5 time steps can be seen in Fig. \ref{fig.StartUpMinimum}.
\end{exam}

\subsection{Economic Dispatch and the Shortest-Path Problem} \label{sec.ED}
The ED problem aims at scheduling the generation of the turbine throughout a time horizon $T$ as to minimize the cost while generating the required amount of heat and power. For each state-control pair $(x,u)$, we define $C_{\rm MGT}(x,u)$ as the total cost of operating the turbine for $c(x,u)$ units of time, starting at state $x$ and issuing the control input $u$. In other words, this is the cost of the transition defined by the state-control pair $(x,u)$. We let $(P(t),H(t))_{t=1}^T$ be the power and heat demand, which are known throughout the time horizon.

Besides the turbine, we can also draw power and heat from a utility. For a time $t$, we denote the power and heat purchased from the utility by $x_U^P(t)$ and $x_U^H(t)$ respectively. The cost of purchasing $x_U^P$ units of power and $x_U^H$ units of heat from the utility at time $t$ is denoted by $C_{U,t}^P(x_U^P),C_{U,t}^H(x_U^H)$ respectively. 
We assume the cost function $C_{U,t}^P$ is defined for $x_U^P$, which corresponds to the case in which the MGT tries to sell power to the utility. For example, a negative cost corresponds to selling power to the utility, and an infinite cost corresponds to inability to sell power. Moreover, we assume the function $C_{U,t}^H$ is defined for $x_U^H<0$ and satisfies $C_{U,t}^H(x_U^H) = 0$. In other words, we can exhaust excess generated heat with no extra cost.
We further assume that the functions $C_{U,t}^P,C_{U,t}^H$ are non-decreasing on the sets $\{x_U^P\in \R : C_{U,t}^P(x_U^P) < \infty\}$ and $\{x_U^H\in \R : C_{U,t}^H(x_U^H) < \infty\}$ respectively, i.e. that buying more power and heat from the utility will cost more, and that selling power to the utility (if possible) will earn more. The ED problem is defined as follows:

\vspace{-15pt}
\small
\begin{align} \label{eq.ED_Orig}
\min ~&~ \sum_{t=1}^{T} \left[C_{\rm MGT}(x(t),u(t)) + C_{U,t}^P(x_U^P(t)) + C_{U,t}^H(x_U^H(t))\right] \\ \nonumber
{\rm s.t.} ~&~ x(t+c(x(t),u(t))\Delta t) = f(x(t),u(t)), ~\forall t=1,\cdots,T\ \\\nonumber
~&~ P_{\rm MGT}(x(t),u(t)) + x_U^P(t) = P(t), ~\forall t\\\nonumber
~&~ H_{\rm MGT}(x(t),u(t)) + x_U^H(t) = H(t), ~\forall t \\\nonumber
~&~ x(t) \in \X,~~ u(t) \in \U(x(t)),~~x_U^P(t),x_U^H(t) \in \mathbb{R}, ~\forall t
\end{align}\normalsize

This problem is evidently a nonlinear mixed-integer problem. However, \cite{Rist2017} offers a quick solution method using a directed graph $\G = (\V,\E)$. The vertices are given by the pairs $(t,x)$ where $t\in\{1,\cdots,T\}$ and $x\in \mathcal{X}$. For a fixed time $t$, the nodes $\{(t,x)\}_{x\in \X}$ designate the state of the turbine at time $t$. As for the edges, $e=(t_1,x_1)\to(t_2,x_2) \in \E$ if there is some $u_1 \in \U(x_1)$ such that $f(x_1,u_1) = x_2$ and $t_2 = t_1 + c(x_1,u_1)$. The cost of said edge is defined as the total cost of the transition, given by the following expression: 
\begin{align} \label{eq.cost}
w_e = C_{\rm MGT}(x_1,u_1) + \sum_{t=t_1}^{t_2} &[C_{U,t}^P (P(t) - P_{\rm MGT}(x_1,u_1)) \\ +& C_{U,t}^H (H(t) - H_{\rm MGT}(x_1,u_1))]. \nonumber
\end{align} 

The edges and their cost represent the possible transitions for the turbine. For example, the corresponding graph for the turbine in Example \ref{eq.StartUpMinimum} with time horizon $T=5$ can be seen in Fig. \ref{fig.StartUpMinimum}. Thus, a possible trajectory $(x(t))_{t=1}^T$ of the state of the turbine corresponds to a \emph{path} in the graph\footnote{See the notations section for a precise definition of a path in a graph.}. If we define the cost of a path as the sum of the costs of the corresponding edges, we get a one-to-one correspondence between paths on the graph $\G$ and generation schedules of the turbine, in which the total cost of a schedule is identical to the cost of the corresponding path. Therefore, the ED problem can be restated as finding the cheapest path from some node $(1,x)$ to some other node $(T,y)$, where $x,y\in\X$ are the initial and final state of the turbine.

Suppose we add a node $s$ (called the source node) and a node $q$ (called the terminal node) to the graph, and add edges $s\to(1,x)$, $(T,x) \to q$ from all $x\in \X$ having zero weight. Any path from some node $(1,x)$ to some other node $(T,y)$, where $x,y\in \X$, uniquely defines a path from $s$ to $q$. Moreover, these paths share the same cost. Thus, the ED problem can be understood as finding the cheapest path from $s$ to $q$, known as the \emph{shortest path problem} \cite{Cormen2009}. If we denote the set of all paths from $s$ to $q$ in $\G$ by ${\rm PATH}_{\rm s \to q}(\G)$, we get the following optimization problem in the variable ${\rm Path}_{s\to q}$:
\begin{align} \label{eq.SPP}
\min_{{\rm Path_{s\to q}}} \left\{\sum_{e\in {\rm Path_{s\to q}}} w_e : \text{${\rm Path_{s\to q}}\in {\rm PATH}_{\rm s \to q}(\G)$}\right\}.
%\min_{{\rm Path_{s\to q}}} ~&~ \sum_{e\in {\rm Path_{s\to q}}} w_e \\ \nonumber
%{\rm s.t.}~&~ \text{${\rm Path_{s\to q}}\in {\rm PATH}_{\rm s \to q}(\G)$}.
\end{align}
As $\G$ is a directed acyclic graph (DAG), standard dynamic programming methods solve this problem quickly, with computational complexity equal to $O(|\E|)$. Moreover, standard graph theory software provides implementation of said methods. Thus, the ED problem for an MGT can be solved quickly using off-the-shelf software.

However, this approach is inapplicable if the demands are unknown and the weights of the edges cannot be determined accurately. Usually an estimate on the demand throughout the horizon is known, so it is tempting to try and solve this problem with the estimate, disregarding the estimation error. However, the optimal solution to many complex real-life optimization problems can perform poorly when the parameters of the problems are changed by even a minuscule amount \cite{BenTal2009,Elsayed2016}. This motivates using tools from robust optimization, giving a bound on the worst-case behavior of a proposed solution.

\subsection{Robust Optimization}
Consider a minimization problem in the variable $x$, where both the cost function $F(x,\xi)$ and constraints $\phi(x,\xi)\le 0$ are affected by an uncertain variable $\xi \in \Xi$, where the inequalities are understood component-wise. In our case, the uncertain variables are the power and heat demands. In classical robust optimization, we choose a subset $\W \subseteq \Xi$ defining all possible values of the uncertainty we consider, coined the \emph{uncertainty set}, and define the worst-case optimization problem \cite{BenTal2009}:
\begin{align*}
\min_{x}~\max_{\xi\in \W} \left\{F(x,\xi)
 : \phi(x,\xi) \le 0,~\forall \xi \in \W\right\}.
%\min_{x} ~&~ \max_{\xi\in \W} \, F(x,\xi)\\
%{\rm s.t.} ~&~ \phi(x,\xi) \le 0,~\forall \xi \in \W.
\end{align*}
This optimization problem assures that the solution is feasible for any value of the uncertainty within the uncertainty set, and gives a bound on its cost.  However, checking that $\phi(x,\xi) \le 0$ for any $\xi \in \W$ is usually very hard or even impossible if $\W$ is infinite. Instead, we reformulate the constraint as $\sup_{\xi \in \W} \phi(x,\xi) \le 0$, which is easier to verify if the supremum can be computed analytically. For example, if $\phi$ is a bi-linear function and $\W$ is defined using finitely many linear inequalities, the constraint can be reformulated to be linear. A common choice for $\W$ is $\W = \{\xi:\ \|\xi\| \le \delta\}$ for some norm $\|\cdot\|$, for which the supremum can be computed using the dual norm, defined as $\|\eta\|_\star = \sup_{\|\xi\| \le 1} \xi^\top \eta$ \cite{BenTal2009}. Common choices for $\|\cdot\|$ are the $p$-norm for $p = 1,2,\infty$, for which the dual norm is $q$-norm with $q = \infty,2,1$ correspondingly. The parameter $\delta$ must be tuned accordingly to avoid over-conservatism as well as over-optimism. See \cite{Bertsimas2004} for more on the implications of choosing a specific uncertainty set.

\subsection{Robust Shortest-Path Problems} \label{sec.Background_RSPP}
The shortest path problem \eqref{eq.SPP} depends on two parameters - the graph $\G$, and the edge weights $\{w_e\}_{e\in \E}$. The robust shortest-path problem studies the case in which the edge weights are unknown, but are assumed to lie inside a set $\W \subseteq \R^{|\E|}$, coined as the ``uncertainty set". More specifically, the problem aims to minimize the worst-case cost:
\begin{align}
\min_{{\rm Path_{s\to q}}} \max_{w\in \W} \left\{\sum_{e\in {\rm Path_{s\to q}}} w_e : \text{${\rm Path_{s\to q}}\in {\rm PATH}_{\rm s \to q}(\G)$}\right\}.
\end{align}
The edge weights are said to be uncorrelated if there exist sets $\W_e \subseteq \R$ for all $e\in \E$ such that $\W = \{w\in \R^{|\E|} : w_e \in \W_e\}$, i.e. knowing the cost of some edge does not give any more information about the cost of other edges.
It is known that the robust shortest path problem for a general uncertainty set $\W$ is NP-hard \cite{Yu1998}. For that reason, several works in the literature proposed uncertainty sets $\W$ which lead to tractable problems.

In \cite{Yaman2001,Montemanni2004}, the authors assume edge costs are uncorrelated, and choose an uncertainty set consisting of confidence intervals for the cost without budgeting the uncertainty, i.e., the uncertainty set was taken as $\W = \{w\in \R^{|\E|} : {\rm lower}_e\le w_e \le {\rm upper}_e$. Hence, the achieved solution is also robust against the case in which all edges incur the maximum possible cost, rendering it vastly overconservative for real-world scenarios in many applications.
This issue is addressed in \cite{Bertsimas2003}, in which a budgeted uncertainty set is considered by bounding the amount of edges whose cost can be different than the nominal value. This method cannot be applied to the problem of ED, as the costs of the edges are demand-dependent, and in practice, the demand will be different from our estimate at any time step, even if by a small amount. 
More recent works consider either a more complex uncertainty budgeting mechanism \cite{Gabrel2013}, or a more sophisticated robustification method \cite{Zhang2018a}. However, the former can yield NP-hard problems, while the latter yields problems which take a long amount of time to solve in practice, even for small graphs with only hundreds of nodes \cite{Zhang2018a}. For ED, the problem in \cite{Rist2017} is converted to a shortest-path problem on a graph with roughly $250,000$ nodes, rendering the approach of \cite{Zhang2018a} as inapplicable.
Moreover, all of these methods assume that the edge costs are either uncorrelated or negatively correlated with each other (i.e., if the cost $w_{e_1}$ significantly deviates from its mean, then the cost $w_{e_2}$ is less likely to deviate from its mean).
However, in shortest-path problems inspired by economic dispatch, e.g. in \cite{Rist2017}, the costs of edges corresponding to the same time step are positively correlated, as both are determined by the demand at the corresponding time step, and a larger demand leads to a larger cost. These reasons motivate the derivation of the algorithms presented in Section \ref{sec.RobustDispatch}.

\section{Robust Economic Dispatch with Uncertain Demands} \label{sec.RobustDispatch}
Consider an ED problem of the form \eqref{eq.ED_Orig}, where the demand $\xi = (P(t),H(t))_{t=1}^T$ is assumed to be unknown. Assume further that the true demand profile is contained in a set $\W \subseteq \R^{2T}$.  Note that in the ED problem, the turbine variables $x(t),u(t)$ must be scheduled in advance, after which the true demand is revealed and the utility variables $x_U^P(t),x_U^H(t)$ are computed from the power- and heat-balance equations, $x_U^P(t) = P(t) - P_{\rm MGT}(x(t),u(t))$ and $x_U^H(t) = H(t) - H_{\rm MGT}(x(t),u(t))$. In other words, the turbine variables $x(t),u(t)$ are treated as initial decision variables, and $x_U^P(t),x_U^H(t)$ are therefore viewed as recourse variables. We use the graph-based interpretation of the problem. For every edge $e \in \E$, we let $w_{e}(\xi)$ be equal to \eqref{eq.cost}, where $\xi = (P(t),H(t))_{t=1}^T$.
%\todo[inline]{But \eqref{eq.cost} is not a function of the uncertainty variable $\xi$ - don't you need to explicitly define how the uncertainty enters the cost structure? or rather are the variables $P(t)$ and $H(t)$ now the uncertain $\xi$?  so maybe you can make that more explicit here and rewrite \eqref{eq.cost} as explicit function of $\xi$ \textcolor{blue}{$\xi$ is defined in the first sentence of this section as $\xi = (P(t),H(t))_{t=1}^T$.}}The robust ED problem turns into a robust version of the shortest path problem \cite{Cormen2009}, which will be abbreviated as (RSPP):
\begin{align} \label{eq.RSPP}
\min_{{\rm Path_{s\to q}}} ~&~ \max_{\xi \in \W} \sum_{e\in {\rm Path_{s\to q}}} w_e(\xi) \tag{RSPP} \\ \nonumber
{\rm s.t.}~&~ \text{${\rm Path_{s\to q}} \in {\rm PATH}_{s\to q}(\G)$}.
\end{align}
%\todo[inline]{The RSPP is stated a bit strangely.  It seems like you are introducing a notation and calling it a constraint?  ${\rm Path_{s\to q}}$ is a path from $s$ to $q$, so you are minimizing over all such paths, subject to what? \textcolor{blue}{There are no additional constraints - I'm minimizing over all possible paths from $s$ to $q$. I added the standard shortest path problem above to acquaint the reader with the notation better.} the st part seems redundant? \textcolor{blue}{I did not understand what you mean here.}
%see my comment again by (3)...
%\textcolor{blue}{I hope it is now better.}
%}
The main focus of this section is to study the tractability of \eqref{eq.RSPP} as a consequence of the choice of $\W$. 
\subsection{Positively-Extreme Profiles and $\mathcal{L}_\infty$-based Uncertainty}
The tractability of \eqref{eq.RSPP} boils down to the following question - what demand profiles $\xi$ are the worst-case for a specific path in the graph $\G$? Intuitively, the higher the demand, the higher the generation cost. It is easy to see by \eqref{eq.cost} that if $P_1(t) \le P_2(t)$ and $H_1(t)\le H_2(t)$ for all $t$, then $w_e(\xi_1) \le w_e(\xi_2)$ for every $e\in \E$, where $\xi_i = (P_i(t),H_i(t))_{t=1}^T$ for $i=1,2$.
%\todo[inline]{I guess here is the first time you state that $\xi_i = (P_i(t),H_i(t))_{t=1}^T$...it should be done earlier to make more sense of this discussion \textcolor{blue}{It was also defined earlier, in the beginning of Section III}.  Also, notation is getting a bit confusing, in another earlier part you write something like $\{P_i(t),H_i(t)\}_{t=1}^T$ (persists later too) \textcolor{blue}{fixed - it's now $(P(t),H(t))_{t=1}^T$ everywhere.}; also dependence on $t$ is not always clear \textcolor{blue}{I did my best to make it clear}}
Thus, for any $\xi_1,\xi_2 \in \R^{2T}$, we have:
\begin{align}\label{eq.Monotonicity}
 (\xi_1)_k \le (\xi_2)_k,~ \forall k=1,\cdots,2T \implies w_e(\xi_1) \le w_e(\xi_2).
\end{align}
This suggests the following definition:
\begin{defn}
Let $\W \subseteq \R^{2T}$ be any set. We say that $\xi \in \W$ is \emph{positively extreme} if for all $\zeta \in \W$ there exists some $k$ such that $\xi_k > \zeta_k$. In other words, we cannot find a point in $\W$ whose entries are all bigger than $\xi$'s. The collection of all positively extreme points in $\W$ will be denoted as $\pe(\W)$.
\end{defn}
%\todo[inline]{the wording here is hard to swallow, can you state it "opposite"?  "We say $\xi \in \W$ is \emph{positively extreme} if all $\zeta \in \W, \zeta \neq \xi$ satisfy $\xi_i \ge \zeta_i$ for all $i$."?? \textcolor{blue}{The definition has been revised. I hope it is better now.}}
\begin{exam}
If $\W = \{\xi\in\R^{2T}: \max_i a_i|\xi_i| \le \mu\}$ then ${\rm pe}(\W)$ contains only the point $(\frac{\mu}{a_1},\cdots,\frac{\mu}{a_{2T}})$.
\end{exam}
\begin{exam} \label{exam.L1}
If $\W = \{\xi\in\R^{2T}: \sum_i a_i|\xi_i| \le \mu\}$, ${\rm pe}(\W)$ contains all points $\xi$ such that $\xi_i \ge 0$ and $\sum_i a_i\xi_i = \mu$. In particular, $\pe(\W)$ is infinite.
\end{exam}

\begin{thm} \label{thm.pe}
Let $\W$ be any bounded closed subset of $\R^{2T}$, and assume all functions $w_e$ satisfy \eqref{eq.Monotonicity}. The problem \eqref{eq.RSPP} for $\W$ is equivalent to the problem \eqref{eq.RSPP} for ${\rm pe}(\W)$, i.e., \small
\begin{align}
\min_{{\rm Path_{s\to q}}} \max_{\xi \in \W} \hspace{-5pt}\sum_{e\in {\rm Path_{s\to q}}}\hspace{-5pt} w_e(\xi) = \min_{{\rm Path_{s\to q}}} \max_{\xi \in {\rm pe}(\W)}\hspace{-5pt} \sum_{e\in {\rm Path_{s\to q}}}\hspace{-5pt} w_e(\xi).
\end{align}\normalsize
\end{thm}
\begin{proof}
Take any path ${\rm Path_{s\to q}}$ from $s$ to $q$, and let $e_1,\cdots,e_\ell$ be its edges. We show that $\max_{\xi \in \W} \sum_{i=1}^\ell w_{e_i}(\xi) = \max_{\xi \in \pe(\W)} \sum_{i=1}^\ell w_{e_i}(\xi)$. Take some $\zeta \in \W$. We claim that there exists a point $\xi \in \pe(\W)$ such that $\zeta_i \le \xi_i$ for all $i$. Indeed, this is true because the set $\W \cap \{\xi: \zeta_i \ge \xi_i\}$ is also bounded and closed, hence it has a positively-extreme point, which must be in $\pe(\W)$ by definition. In particular, we conclude by \eqref{eq.Monotonicity} that
\begin{align*}
 \sum_{i=1}^\ell w_{e_i}(\zeta) \le  \sum_{i=1}^\ell w_{e_i}(\xi) \le \max_{\xi \in \pe(\W)} \sum_{i=1}^\ell w_{e_i}(\xi).
\end{align*}
Maximizing over $\zeta \in \W$ completes the proof.
\end{proof}
%\todo[inline]{not a big deal, but a little confusing - in definition you are talking about $\xi$ being positively extreme, but then in proof you use the variable $\zeta$ for pe points... \textcolor{blue}{Fixed.}}
We now examine a corollary of Theorem \ref{thm.pe} that considers an estimate for the power and heat demand at a time $t$, denoted by  $P_0(t),H_0(t)$ respectively, and an estimation error that is bounded by variables $\Delta P(t),\Delta H(t)$ respectively.
\begin{cor} \label{cor.Linfty}
Suppose that the set $\W$ is given by:
\begin{align} \label{eq.Linfty}
\mathcal{W} = \left\{\left(P(t),H(t)\right)_{t=1}^T: {|P(t)-P_0(t)| \le \Delta P(t),\atop |H(t) - H_0(t)| \le \Delta H(t)}\right\},
\end{align}
%\textcolor{blue}{i.e., the estimate for the power and heat demand at a time $t$ is given by $P_0(t),H_0(t)$ respectively, and the estimation error is bounded by $\Delta P(t),\Delta H(t)$ repsectively.}
The robust ED problem with uncertainty set $\W$ is equivalent to the ED problem with demand $P(t) = P_0(t) + \Delta P(t)$, $H(t) = H_0(t) + \Delta H(t)$. Thus, it  can be solved in $O(|\E|) = O(\max_{x,u} c(x,u) |\X| T)$ time.
\end{cor}
%\todo[inline]{$P_0(t)$ and $H_0(t)$ should be defined...perhaps with some setup before corollary \textcolor{blue}{Added.}}
\begin{proof}
It's enough to show that \eqref{eq.RSPP} for the set $\W$ is equivalent to the shortest path problem with weights $w_e(\xi)$ for $\xi = (P_0(t) + \Delta P(t),H_0(t)+\Delta H(t))_{t=1}^T$. This follows immediately from Theorem \ref{thm.pe} and the fact that by definition, $\pe(\W) =  (P_0(t) + \Delta P(t),H_0(t)+\Delta H(t))_{t=1}^T$. Solving the shortest path problem in a DAG takes $O(|\E|)$ time \cite{Cormen2009}.
\end{proof}

The corollary above shows that the robust ED problem can be solved in a tractable manner if $\W$ has the form \eqref{eq.Linfty}, as it is equivalent to a shortest-path problem. However, in \eqref{eq.Linfty}, the demand is merely assumed to be within given confidence intervals for each time step. This assumption might lead to mediocre results in practice - if the confidence intervals are taken too large, the solution may be over-conservative, and if they are taken too small, we do not account for unforeseen short demand spikes. A common way to deal with this problem is to use uncertainty sets which also specify the $1$-norm, i.e. they budget the uncertainty over all time steps. This will be the focus of the next subsection. 

Before moving forward, we want to return to Example \ref{exam.L1}. There, $\pe(\W)$ was infinite, meaning that the problem $\max_{\xi \in \pe(\W)} \sum w_{e}(\xi)$ is hard to solve, unless more assumptions are added. If we assume the functions $w_e$ are convex in $\xi$, the maximized function is also convex, so the maximum is attained at an extreme point of the set $\pe(\W)$ \cite[Theorem 32.2]{Rockafellar1970}. The convexity of the functions $w_e$ can be understood using the convexity of the functions $C_{U,t}^P,C_{U,t}^H:$
\begin{prop}
All of the functions $\{w_e\}_{e\in \E}$ are convex if and only if all of the functions $C_{U,t}^P,C_{U,t}^H$ are convex
\end{prop}
\begin{proof}
We fix an edge $e$ from a node $(t_1,x_1)$ to a node $(t_2,x_2)$, where $x_1,x_2\in \X$ are states of the turbine. Because there exists an edge between $(t_1,x_1)$ and $(t_2,x_2)$, there exists a control action $u_1 \in \U(x_1)$ such that $f(x_1,u_1) = x_2$ and $t_2 = t_1 + c(x_1,u_1)\Delta t$. Recall that $w_e$ was defined as a function of $\xi = (P(t),H(t))_{t=1}^T$ using the following expression:
\begin{align*}
    w_e(\xi) = C_{\rm MGT}(x_1,u_1) + \sum_{t=t_1}^{t_2}[&C_{U,t}^P (P(t) - P_{\rm MGT}(x_1,u_1)) \\+& C_{U,t}^H (H(t) - H_{\rm MGT}(x_1,u_1))]
\end{align*}
The result now follows from the fact that $C_{\rm MGT}(x_1,u_1), P_{\rm MGT}(x_1,u_1)$ and $H_{\rm MGT}(x_1,u_1)$ are all constant with respect to $\xi$.
\end{proof}

For that reason, we make the following assumption:
\begin{assump} \label{assump.Convexity}
For every time $t$, the utility cost functions $C_{U,t}^P,C_{U,t}^H$ are convex. Equivalently, the cost functions $w_e$ are convex.
\end{assump}
\begin{rem}
The convexity of $C_{U,t}^P$ can be easily deduced for many cases. For example, the cost function implemented in European electricity markets is a linear function, in which the per-unit price is achieved by an optimization problem aggregating all the demands and generations in the network \cite{Martin2014}. In other cases, service operators explicitly convexify this cost function \cite{Schiro2016}. 
Alternatively, utility operators put a fixed per-unit cost, as well as fixed costs and demand charges which only go into effect if the demand is positive \cite{Beaudin2015}. In this case, if we cannot sell power back to the utility, then $C_{U,t}^P(x_U^P) = A_tx_U^P+B_t$ for some possibly time-dependent parameters $A_t,B_t$ and $x_U^P \ge 0$, and $C_{U,t}^P(x_U^P) =\infty$ if $x_U^P<0$. Thus, $C_{U,t}^P$ is convex. If we instead consider the case in which there is only a fixed per-unit cost and a fixed cost, then $C_{U,t}^P$ is affine and thus convex.

For heat, most consumers use a boiler to satisfy their heat demand, in which case one can model the heating cost with the price of natural gas. In that case, the cost function $C_{U,t}^H$ is given by $$C_{U,t}^H(x_U^H) = \begin{cases} Ax_U^H, & x_U^H \ge 0 \\ 0, & x_U^H \le 0 \end{cases},$$ as excess heat can be exhausted with no extra cost. In particular, $C_{U,t}^H$ is convex. See \cite{Rist2017} for more details.

In any case, if either $C_{U,t}^P$ or $,C_{U,t}^H$ is not convex, and Assumption \ref{assump.Convexity} is needed, we approximate them by convex functions. We thus yield a suboptimal solution for \eqref{eq.RSPP}, whose quality depends on the approximation error of $C_{U,t}^P,C_{U,t}^H$.
\end{rem}

%\todo[inline]{can we provide some citations supporting such assumption? \textcolor{blue}{Done. (Hopefully)} well - lets ask beni :) }
%\todo[inline]{setup corollary with text. Also, below corollary should have a proof, even if it is a citation stating that solutions for shortest path occur at extreme points. \textcolor{blue}{Done.}}

Under Assumption \ref{assump.Convexity}, we can prove an analogue of Theorem \ref{thm.pe} which relies on the notion of extreme points:
\begin{cor} \label{cor.ExtPe}
Let $\W \subseteq \mathbb{R}^2$ be bounded and closed, and assume $w_e$ satisfies \eqref{eq.Monotonicity} and Assumption \ref{assump.Convexity}. The problem \eqref{eq.RSPP} for $\W$ is equivalent to the problem \eqref{eq.RSPP} for $\ext(\pe(\W))$.
\end{cor}

\begin{proof}
 Fix any path ${\rm Path_{s\to q}}$ from $s$ to $q$, and let $e_1,\cdots,e_\ell$ be its edges. We show that $\max_{\xi \in \W} \sum_{i=1}^\ell w_{e_i}(\xi) = \max_{\xi \in \ext(\pe(\W))} \sum_{i=1}^\ell w_{e_i}(\xi)$. 
First, by Theorem \ref{thm.pe}, we have that $\max_{\xi \in \W} \sum_{i=1}^\ell w_{e_i}(\xi) = \max_{\xi \in \pe(\W)} \sum_{i=1}^\ell w_{e_i}(\xi)$.
Second, we note that $\sum_{i=1}^\ell w_{e_i}(\xi)$ is a convex function in $\xi$, so by \cite[Theorem 32.2]{Rockafellar1970}, for any closed bounded set $\mathcal{A}$ we have that 
\begin{align*}
\max_{\xi \in \mathcal{A}} \sum_{i=1}^\ell w_{e_i}(\xi) = \max_{\xi \in \ext(\mathcal{A})} \sum_{i=1}^\ell w_{e_i}(\xi).
\end{align*}
Choosing $\mathcal{A} = \pe(\W)$ completes the proof.
\end{proof}

\begin{exam} \label{exam.L1_2}
If $\W = \{\xi\in\R^{2T}: \sum_i a_i|\xi_i| \le \mu\}$, then $\ext(\pe(\W)) = \{\xi^{(1)},\cdots,\xi^{(2T)}\}$, where $\xi^{(j)}_i = \frac{\mu}{a_i}\mathds{1}_{\{i = j\}}$.
\end{exam}

\subsection{Mixed $\mathcal{L}_1\setminus\mathcal{L}_\infty$ Uncertainty}
For this subsection, we assume Assumption \ref{assump.Convexity} holds. We want to consider an uncertainty set $\W$ including both unforeseen short demand spikes as well as a constant bias from the estimate. A natural choice here is:
\begin{align*}
\mathcal{W} = \left\{(P(t),H(t))_{t=1}^T: \substack{|P(t)-P_0(t)| \le \Delta P(t),\\\substack{|H(t) - H_0(t)| \le \Delta H(t)\\ \sum_{t=1}^T \left[\frac{|P(t)-P_0(t)|}{\Delta P(t)}+\frac{|H(t)-H_0(t)|}{\Delta H(t)}\right] \le \mu}}\right\}.
\end{align*}
However, it is possible to show that unless $\frac{\mu}{2T} \ll 1$ or $1-\frac{\mu}{2T} \ll 1$, $|\ext(\pe(\W))|$ is exponential in $T$. Thus, Corollary \ref{cor.ExtPe} will not yield a tractable optimization problem. Instead, we consider a different uncertainty set:
\begin{align} \label{eq.MixedNorm}
\mathcal{W} = \left\{(P(t),H(t))_{t=1}^T: \substack{P(t) = P_0(t) + \eta_1^P(t) + \eta_\infty^P(t),\\ H(t) = H_0(t) + \eta_1^H(t) + \eta_\infty^H(t),\\
\sum_{t=1}^T \left[|\delta_{P,t}\eta^P_1(t)| + |\delta_{H,t}\eta^H_1(t)|\right] \le \mu_1,\\
 |\eta^P_\infty(t)|\le \Delta P(t), |\eta^H_\infty(t)|\le \Delta H(t),~\forall t}\right\},
\end{align}
where $\mu_1,\Delta P(t),\Delta H(t),\delta_{P,t},\delta_{H,t} \ge 0$ are parameters. 
This uncertainty set will be called the \emph{``mixed $\mathcal{L}_1$/$\mathcal{L}_\infty$ uncertainty set."} Intuitively, it dissects the uncertainty in the demand into two factors - the first, $\eta^P_1(t),\eta^H_1(t)$, corresponds to large-but-short unforeseen demand spikes, and the second, $\eta^P_\infty(t),\eta^H_\infty(t)$, corresponds to a small-but-long bias from the estimated demand, $(P_0(t),H_0(t))_{t=1}^T$. If $\mu_1,\Delta P(t),\Delta H(t),\delta_{P,t},\delta_{H,t}$ are tuned correctly, the uncertainty set can model both without being too conservative. Similarly to Example \ref{exam.L1}, the set $\ext(\pe(\mathcal{U}))$ consists of $2T$ elements, $(\bar{P}(t)+\Delta P^{(i)}(t),\bar{H_0}(t))_{t=1}^T$ and $(\bar{P}_0(t),\bar{H}_0(t)+\Delta H^{(i)}(t))_{t=1}^T$ for $i=1,\cdots,T$ , where:
\begin{align}\label{eq.Aux}
&\bar{P}_0(t) = P_0(t)+\Delta P(t),~
\bar{H}_0(t) = H_0(t)+\Delta H(t).\\\nonumber
&\Delta P^{(i)}(t) = \frac{\mu_1}{\delta_{P,t}}\mathds{1}_{\{i=t\}},~
\Delta H^{(i)}(t) = \frac{\mu_1}{\delta_{H,t}}\mathds{1}_{\{i=t\}}.
\end{align}
The demands $\bar{P_0}(t)$ and $H_0(t)$ serve as a worst-case scenario if there are no demand spikes, similarly to Corollary \ref{cor.Linfty}. For each $i = 1,2,\cdots,T$, the terms $\Delta P^{(i)}, \Delta H^{(i)}$ correspond to the highest possible demand spike at time $i$. In particular, for any $i$, the sum of these terms serves as a possible worst-case scenario for the uncertainty set \eqref{eq.MixedNorm}.
Thus, the robust ED problem for $\W$ is reduced to \eqref{eq.RSPP} with $2T$ possible demand profiles - $2$ demand profiles for each time $i$, in which the power or heat demand spikes at time $i$, respectively. We want to use an augmented form of the shortest-path algorithm to solve this problem. In order to do so, we first reformulate it as a problem closer to the classical shortest-path problem.

For each edge $e = (t_1,x_1) \to (t_2,x_2) \in \E$, we consider the cost $w_e$ for all $2T$ possible scenarios. We first define $W_{bias}(e)$ as $w_e((\bar{P}_0(t),\bar{H}_0(t)))$, which is the cost on the edge $e$ corresponding the worst-case spikeless demand, stemming only from the long-but-small bias in demand, which must be paid for each of the extreme $2T$ scenarios, no matter when the spike occurs. We also let $W_{\rm spike}(e)$ be the highest possible additional cost stemming from unforeseen demand spikes, defined as 
\small
\begin{align}\label{eq.W_add}
\max_{t_1 \le k < t_2} \{&w_e((\bar{P}_0(t) + \Delta P^{(k)}(t),\bar{H}_0)(t)),\\&w_e((\bar{P}_0(t),\bar{H}_0(t) + \Delta H^{(k)}(t)))\} - W_{\rm bias}(e),\nonumber
\end{align}\normalsize
That is, $W_{\rm spike}(e)$ is defined as the maximum possible cost of the edge $e$ in any of the $2T$ possible scenarios, minus $W_{\rm bias}(e)$.
 Given a path ${\rm Path_{s\to q}}$ from $s$ to $q$, the cost function $\max_{\xi \in \W} \sum_{e\in {\rm Path_{s\to q}}} w_e(\xi)$ is the sum of $\sum_{e\in {\rm Path_{s\to q}}}W_{\rm bias}(e)$ and $\max_{e\in {\rm Path_{s\to q}}}W_{\rm spike}(e)$. Indeed, if the demand profile is equal to $(\bar{P}_0(t) + \Delta P^{(k)}(t),\bar{H}_0(t))_{t=1}^T$ or to $(\bar{P}_0(t),\bar{H}_0(t) + \Delta H^{(k)}(t))_{t=1}^T$, meaning there is a spike in demand at time $k$, then all edges corresponding to transitions outside time $k$ only have the ``bias" cost, while the single edge corresponding to a transition at time $k$ has a $W_{\rm bias} + W_{\rm spike}$. Moreover, given any edge $e = (t_1,x_1) \to (t_2,x_2) \in \E$, there is at least one $k$ for which this edge has an additional cost equal to $W_{\rm spike}$. Thus, the optimization problem becomes:
\begin{align} \label{eq.MixedSP}
\min ~&~ \sum_{e\in {\rm Path_{s\to q}}} W_{\rm bias}(e) + \max_{e\in {\rm Path_{s\to q}}} W_{\rm spike}(e) \\ \nonumber
{\rm s.t.}~&~ \text{${\rm Path_{s\to q}} \in {\rm PATH}_{s\to q}(\G)$}.
\end{align}

We would like to consider an algorithm for solving the problem \eqref{eq.RSPP} in this case. A key idea that will be used is to consider the graph $\G$ with a different set of weights. For each number $\alpha \in \mathbb{R}$, we define the $\alpha$-restricted graph $\G_\alpha$ as a weighted graph $(\V,\E,\omega_\alpha)$, where $\G = (\V,\E)$ and $\omega_\alpha(e) = \begin{cases} W_{\rm bias}(e) & W_{\rm spike}(e) \le \alpha \\ \infty & W_{\rm spike}(e) > \alpha \end{cases}.$
We present the following algorithm for solving the problem. First, we compute all parameters as in \eqref{eq.Aux}, and define $W_{\rm bias},W_{\rm spike}$ as above. Second, we define an array of thresholds called ${\rm Thresh}$. For each threshold $\alpha \in {\rm Thresh}$, we solve the classical shortest-path problem on the weighted graph $\mathcal{G}_\alpha$. We let $V_{\rm path}(\alpha)$ be this shortest path on $\mathcal{G}_\alpha$, and let $V_{\rm cost}(\alpha)$ as the total cost of this path in the problem \eqref{eq.MixedSP}. We then choose the return the path $V_{\rm path}(\alpha)$ for which $V_{\rm cost}(\alpha)$ is minimal over all thresholds $\alpha$. We show that choosing the set of thresholds as $\{W_{\rm spike}(e)\}_{e\in \E}$ guarantees that we achieve an optimal solution of \eqref{eq.MixedSP}.

%\todo[inline]{here you should probably summarize in words what the algorithm is doing - including introducing some of the new notations that only appear in algorithm before showing up again in theorem proof \textcolor{blue}{added.}}
\begin{algorithm}
\caption{Optimal Economic Dispatch for Mixed $\mathcal{L}_1$/$\mathcal{L}_\infty$ Uncertainty}
{\bf Input:} An uncertainty set $\mathcal{W}$ of the form \eqref{eq.MixedNorm}.\\
{\bf Output:} An optimal solution to the corresponding robust economic dispatch problem.
\label{alg.MixedL1L2}
\begin{algorithmic}[1] 
\State Define $\bar{P}_0(t),\bar{H}_0(t),\Delta P^{(i)}(t),\Delta H^{(i)}(t)$ as in \eqref{eq.Aux}.
\State Define four arrays $W_{\rm bias},W_{\rm spike},V_{\rm cost},V_{\rm path}$.
\For{each edge $e$ in the graph}
\State Define $W_{\rm bias}(e) = w_e((\bar{P}_0(t),\bar{H}_0(t))_{t=1}^T)$.
\State Define $W_{\rm spike}(e)$ as in \eqref{eq.W_add}.
\EndFor
\State \label{step.7} Define the array ${\rm Thresh} = W_{\rm spike}$.
\For {$\alpha \in {\rm Tresh}$}
\State Solve the shortest-path problem from $s$ to $q$ for $\G_\alpha$. Let ${\rm Path}_{s\to q}$ the optimal path. If there's a tie, favor the path with the lower maximal $W_{\rm spike}$.
\State Define $V_{\rm path}(\alpha) = {\rm Path}_{s\to q}$.
\State Define $V_{\rm cost}(\alpha) = \sum_{e\in {\rm Path}_{s\to q}} W_{\rm bias}(e) +  \max_{e\in {\rm Path}_{s\to q}}W_{\rm spike}(e)$, the cost of the path ${\rm Path}_{s\to q}$ for the problem \eqref{eq.MixedSP}.
\EndFor
\State \label{step.13} Find $\gamma = \arg\min_\beta V_{\rm cost}(\beta)$.
\State {\bf return:} $V_{\rm cost}(\gamma), V_{\rm path}(\gamma)$.
\end{algorithmic}
\end{algorithm}

\begin{thm} \label{thm.L1Linfty}
Algorithm \ref{alg.MixedL1L2} solves \eqref{eq.RSPP} for the uncertainty set $\W$ of the form \eqref{eq.MixedNorm}, with computational complexity $O(|\E|^2)$.% Moreover, it's computational complexity is $O(|\E|^2)$. 
\end{thm}

\begin{proof}
Let ${\rm Path}^\star_{s\to q}$ be the optimal solution of \eqref{eq.RSPP}. By optimality, for any other path ${\rm Path}_{s\to q}$ at least one of the following holds:
\begin{align}
\sum_{e\in {\rm Path^{\star}}_{s\to q}} W_{\rm bias}(e) &\le \sum_{e\in {\rm Path}_{s\to q}} W_{\rm bias}(e), \label{eq.cond1} \\
\max_{e\in {\rm Path^{\star}}_{s\to q}} W_{\rm spike}(e) &\le \max_{e\in {\rm Path}_{s\to q}} W_{\rm spike}(e), \label{eq.cond2}
\end{align}
where if at least one holds with equality, then both inequalities hold. Consider the graph $G_\alpha$ for $\alpha = \max_{e\in {\rm Path^{\star}}} W_{\rm spike}(e)$. For this graph, if a path ${\rm Path}_{s\to q}$ satisfies the inequality \eqref{eq.cond2}, it does so with equality. Thus, it must satisfy \eqref{eq.cond1}, meaning that ${\rm Path^{\star}}$ is the shortest path from $s$ to $q$ in $\G_\alpha$. By the tie break rule and conditions \eqref{eq.cond1}, \eqref{eq.cond2}, we conclude that $V_{\rm path}(\alpha) = {\rm Path^{\star}}$. Moreover, for any path $V_{\rm path}(\beta)$, we have:
\begin{align*}
V_{\rm cost}(\beta) =  \sum_{e\in V_{\rm path}(\beta)} W_{\rm bias}(e) + \max_{e\in V_{\rm path}(\beta)} W_{\rm spike}(e)
\end{align*}
Thus, by optimality of ${\rm Path^{\star}}$, $\alpha = \arg\min_\beta V_{\rm cost}(\beta)$, and the returned path is ${\rm Path^{\star}}$. This proves the correctness of the algorithm.
As for the time complexity, the parts outside the for-loop on $\alpha$ take $O(|\E|)$ time. Inside the for-loop, we build the DAG $\G_\alpha$, which takes $O(|\E|)$ time, and solve the shortest-path problem on it, which also takes $O(|\E|)$ time. The for-loop has $O(|\E|)$ iterations, so get a time complexity of $O(|\E|^2)$.
\end{proof}

\begin{rem}
\textcolor{black}{
The complexity of graph-based algorithms is traditionally written in terms of the number of vertices $|\V|$ in the graph, or the number of edges $|\E|$ in the graph. The complexity estimate of Theorem \ref{thm.L1Linfty} follows this norm, writing the complexity as $O(|\E|^2)$. However, we would like to connect this complexity estimate to the turbine model.}

\textcolor{black}{To do so, for any $x\in \X$, we let $\rho(x) = |\U(x)|$ be the  number of possible control actions at $x$. We also let $\rho_\X = \frac{1}{|\X|} \sum_{x\in \X}\rho(x)$ be the average number of control actions at a state. If the time horizon $T$ is long enough, then $|\E| = O(T|\X|\rho_\X)$. For example, this happens if $T$ is at least twice as long as the longest transition, $\max_{x,u} c(x,u)$. In that case, the complexity estimate is  $O(\E^2) = O(T^2|\X|^2 \rho_\X^2)$. This estimate will also be helpful later, when we present more efficient algorithms with complexity $O(|\E|) = O(T|\X|\rho_\X)$.}
\end{rem}

\begin{rem} \label{rem.ReducedComplexity}
The complexity bound $O(|\E|^2)$ can sometimes be too high for real-world economic dispatch problems, in which the graph $\G$ can have more than a million edges \cite{Rist2017}. Instead, we note that if $\delta_{P,t},\delta_{H,t}$ defined in \eqref{eq.MixedNorm} do not change too often, the array ${\rm Thresh}$ contains many repetitions. If ${\rm Thresh}$ contains $L\le |\E|$ different elements, then the computational complexity of the algorithm is $O(|\E|L)$.
We denote the number of different values that $\delta_{P,t},\delta_{H,t}$ take for $t=1,\cdots T$ as $n_P,n_H$ respectively. It's easy to see that if $T \gg \max_{x,u} c(x,u)$, each possible transition appears about the same number of times, then $L$ scales linearly with $n_P + n_H$. Thus, we get that $L = O\left(\frac{|\E|}{T}(n_P+n_H)\right)$. Therefore, we get an algorithm whose time complexity $O(|\E|L) =  O\left(\frac{|\E|^2}{T}(n_P+n_H)\right)$ grows linearly with the time horizon $T$, as for a fixed state-space representation, $|\E| = O(T)$ holds.
\end{rem}

Algorithm \ref{alg.MixedL1L2}, together with Remark \ref{rem.ReducedComplexity}, give a linear-time algorithm if most $\delta_{P,t},\delta_{H,t}$ have the same value. If this is not the case, we can give a linear-time algorithm that achieves an approximation of the optimal solution. Namely, we prove:
\begin{lem} \label{lem.approximation}
Consider the problem \eqref{eq.RSPP} for the uncertainty set $\W$ of the form \eqref{eq.MixedNorm}, and let ${\rm Path}^\star_{s\to q}$ be the optimal solution, $\alpha = \max_{e\in {\rm Path^{\star}}} W_{\rm spike}(e)$, and let $\beta \ge \alpha$ be any number. Denote ${\rm Path}_{s\to q}^{\beta}$ as the shortest path from $s$ to $q$ in $\G_\beta$. If:
\begin{align*}
 V^\star &=   \sum_{e\in {\rm Path}^\star_{s\to q}} W_{\rm bias}(e) + \max_{e\in {\rm Path}^\star_{s\to q}} W_{\rm spike}(e)\\
 V^\beta &=   \sum_{e\in {\rm Path}^\beta_{s\to q}} W_{\rm bias}(e) + \max_{e\in {\rm Path}^\beta_{s\to q}} W_{\rm spike}(e),
\end{align*}
then $V^\star \le V^\beta \le \min\{V^\star + \beta-\alpha,\frac{\beta}{\alpha} V^\star\}$.
%the following implications hold:
%\begin{enumerate}
%\item If $\epsilon = \beta - \alpha$, then $V^\star \le %V^\beta \le V^\star + \epsilon$\\
%\item If $\tau = \frac{\beta}{\alpha}$, then $V^\star \le %V^\beta \le \tau V^\star$.
%\end{enumerate}
\end{lem}

\begin{proof}
It suffices to show that the right side of the inequality holds.  
First, as $\beta \ge \alpha$, the path ${\rm Path}^\star_{s\to q}$ has a finite cost in the graph $\G_\beta$, equal to its cost in $\G$. Thus, we find that:
\begin{align} \label{eq.FirstIneq}
\sum_{e\in {\rm Path}^\beta_{s\to q}} W_{\rm bias}(e) \le \sum_{e\in {\rm Path}^\star_{s\to q}} W_{\rm bias}(e).
\end{align}
Moreover,
\begin{align} \label{eq.SecondIneq}
\max_{e\in {\rm Path}^\beta_{s\to q}} W_{\rm spike}(e) \le \beta = \alpha + \epsilon \le \max_{e\in {\rm Path}^\star_{s\to q}} W_{\rm spike}(e) + \epsilon.
\end{align}
Summing \eqref{eq.FirstIneq} and \eqref{eq.SecondIneq} gives $V^\beta \le V^\star + \epsilon$. Moreover, as $\tau \ge 1$, we have,
\begin{align} \label{eq.ThirdIneq}
\sum_{e\in {\rm Path}^\beta_{s\to q}} W_{\rm bias}(e) \le \tau\sum_{e\in {\rm Path}^\star_{s\to q}} W_{\rm bias}(e),
\end{align}
and
\begin{align} \label{eq.FourthIneq}
\max_{e\in {\rm Path}^\beta_{s\to q}} W_{\rm spike}(e) \le \beta = \tau\alpha \le \tau\max_{e\in {\rm Path}^\star_{s\to q}} W_{\rm spike}(e).
\end{align}
Summing \eqref{eq.ThirdIneq} and \eqref{eq.FourthIneq} gives $V^\beta \le (1+\mu)V^\star$.
\end{proof}

Lemma \ref{lem.approximation} can be used to prescribe linear-time algorithms approximating the optimal solution of \eqref{eq.RSPP} for the uncertainty set \eqref{eq.MixedNorm}:
\begin{thm} \label{thm.AdditiveApproximation}
Consider Algorithm \ref{alg.MixedL1L2} and take any $\epsilon > 0$. Suppose that we change step \ref{step.7} and define 
\begin{align*}
{\rm Thresh} = \bigg\{&\min_{e\in \E} W_{\rm spike}(e),\min_{e\in \E} W_{\rm spike}(e) + \epsilon,\\&\min_{e\in \E} W_{\rm spike}(e) + 2\epsilon, \cdots, \max_{e\in \E} W_{\rm spike}(e)\bigg\}.
\end{align*}
Let $V^{\star,\epsilon}$ be the value provided by this modified algorithm, and let $V^\star$ be the optimal value of \eqref{eq.RSPP} for the uncertainty set $\W$ of the form \eqref{eq.MixedNorm}. Then $V^\star \le V^{\star,\epsilon} \le V^\star + \epsilon$. Moreover, the computational complexity of the modified algorithm is $O\left(|\E|\cdot \left\lceil\frac{\max_{e\in \E} W_{\rm spike}(e) - \min_{e\in \E} W_{\rm spike}(e)}{\epsilon}\right\rceil\right)$.
\end{thm}
\begin{proof}
Suppose that $\alpha$ is $\max_e W_{\rm spike}(e)$, where the maximum is taken over the optimal solution to \eqref{eq.RSPP}. By construction, there exists some $\beta \in {\rm Thresh}$ such that $\beta \le \alpha$ and $\epsilon \le \beta - \alpha$. By Lemma \ref{lem.approximation}, we conclude that $V^\star \le V_{\rm cost}(\beta) \le V^\star + \epsilon$. By step \ref{step.13}, we have that $V^{\star,\epsilon} \le V_{\rm cost}(\beta) \le V^\star + \epsilon$. The inequality $V^\star \le V^{\star,\epsilon}$ is clear, as $V^\star$ is the optimal cost over all possible trajectories. Thus $V^\star \le V^{\star,\epsilon} \le V^\star + \epsilon$. 
As for the time complexity, the same argument as in the proof of Theorem \ref{thm.L1Linfty} shows that the time complexity is $O(|\E|N)$, where $N$ is the number of points in ${\rm Thresh}$. It can easily be seen that $N = \left\lceil\frac{\max_{e\in \E} W_{\rm spike}(e) - \min_{e\in \E} W_{\rm spike}(e)}{\epsilon}\right\rceil + 1$, which gives the desired complexity bound. This completes the proof.
\end{proof}

Similarly, we prove:
\begin{thm} \label{thm.MultiApprox}
Consider Algorithm \ref{alg.MixedL1L2} and take any $\mu > 0$. Suppose that we change step \ref{step.7} and define 
\begin{align*}
{\rm Thresh} = \bigg\{&\min_{e\in \E} W_{\rm spike}(e),(1+\mu)\min_{e\in \E} W_{\rm spike}(e),\\&(1+\mu)^2\min_{e\in \E} W_{\rm spike}(e), \cdots, \max_{e\in \E} W_{\rm spike}(e)\bigg\}.
\end{align*}
Let $V^{\star,\mu}$ be the value provided by this modified algorithm, and let $V^\star$ be the optimal value of \eqref{eq.RSPP} for the uncertainty set $\W$ of the form \eqref{eq.MixedNorm}. Then $V^\star \le V^{\star,\mu} \le (1+\mu)V^\star$. Moreover, the computational complexity of the modified algorithm is $O\left(|\E|\cdot \left\lceil\frac{\log\max_{e\in \E} W_{\rm spike}(e) - \log\min_{e\in \E} W_{\rm spike}(e)}{\log(1+\mu)}\right\rceil\right)$.
\end{thm}
\begin{proof}
As before, let $\alpha=\max_e W_{\rm spike}(e)$, where thee maximum is taken over the optimal solution to \eqref{eq.RSPP}. By construction, there exists some $\beta \in {\rm Thresh}$ such that $1\le\frac{\beta}{\alpha} \le 1+\mu$. By Lemma \ref{lem.approximation}, we conclude that $V^\star \le V_{\rm cost}(\beta) \le (1+\mu)V^\star$. By step \ref{step.13}, we have that $V^{\star,\mu} \le V_{\rm cost}(\beta) \le (1+\mu)V^\star$. Together with $V^\star \le V^{\star,\mu}$, stemming from optimality, we conclude that $V^\star \le V^{\star,\mu} \le (1+\mu)V^\star$. 
As for the time complexity, the same argument as in the proof of Theorem \ref{thm.L1Linfty} shows that the time complexity is $O(|\E|N)$, where $N$ is the number of points in ${\rm Thresh}$. It can easily be seen that $N = \left\lceil\frac{\log\max_{e\in \E} W_{\rm spike}(e) - \log\min_{e\in \E} W_{\rm spike}(e)}{\log(1+\mu)}\right\rceil + 1$, which gives the desired complexity bound.
\end{proof}

\subsection{Discussion about Uncertainty Sets and Algorithms} \label{subsec.L1LinftyBalls}
In the previous sections, we presented two possible choices for the uncertainty set. The first one, \eqref{eq.Linfty}, can be understood as a norm ball of a weighed $\ell^\infty$-norm, centred around the point $(P_0(t),H_0(t))$. Indeed, the condition in \eqref{eq.Linfty} can be restated as $\frac{1}{\Delta P(t)}|P(t) - P_0(t)| \le 1$ and $\frac{1}{\Delta H(t)}|H(t) - H_0(t)| \le 1$, which define a norm ball of radius $1$ around the point $(P_0(t),H_0(t))$. The second choice of uncertainty set, \eqref{eq.MixedNorm}, can be similarly seen as a Minkowski sum of two norm balls centred around $(P_0(t),H_0(t))$, the first being a weighted $\ell^1$-norm, and the second being a weighted $\ell^\infty$-norm. 

The tunable parameters $\Delta P(t),\Delta H(t), \delta_{P,t},\delta_{H,t}, \mu_1$ are used to determine the size and exact shape of these uncertainty set. Figure \ref{fig.UncertaintySets} demonstrates the $\ell^1$-normed ball, $\ell^\infty$-normed ball, and their Minkowski sum in $\R^2$. It is seen in Figure \ref{fig.UncertaintySets} that with correct scaling, the Minkowski sum is a subset of an $\ell^\infty$-normed ball which does not contain points in which all entries are ``large" in absolute value, but still contains points in which a subset of the entries is large.

\begin{figure} [!ht] 
    \centering
    \includegraphics[width = 0.4\textwidth]{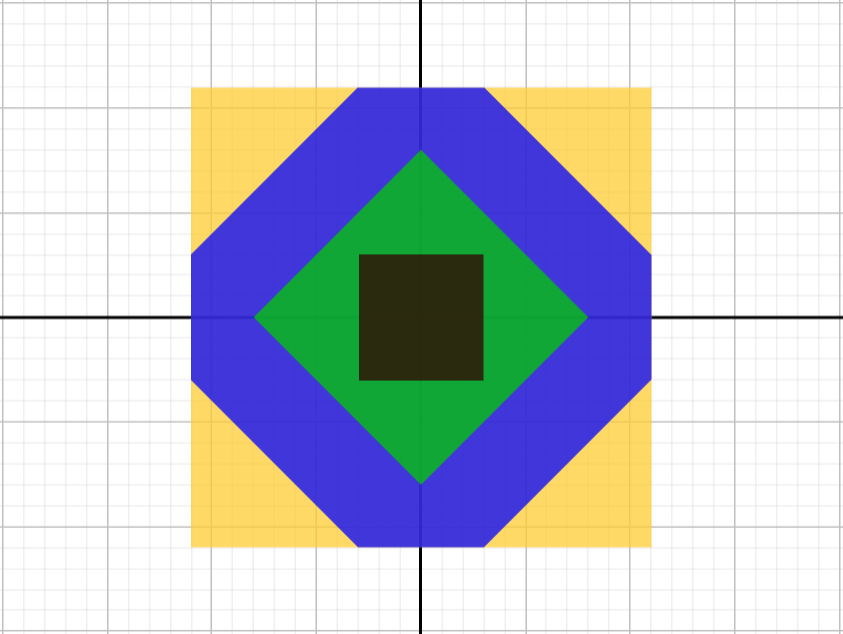}
    \caption{Normed ball in $\R^2$. The black set is an $\ell^\infty$-normed ball, and the green set is an $\ell^1$-normed ball. The yellow set is a larger $\ell^\infty$-normed ball, which is the analogue of \eqref{eq.Linfty}. The blue set is the Minkowski sum of the black and purple sets, which is the analogue of \eqref{eq.MixedNorm}. The blue set is a subset of the yellow set which does not include points in which both entries are large in absolute value.}
   % \vspace{-15pt}
    \label{fig.UncertaintySets}
\end{figure}

We now return to the uncertainty sets \eqref{eq.Linfty} and \eqref{eq.MixedNorm}. With correct tuning, the set \eqref{eq.MixedNorm} is a subset of \eqref{eq.Linfty} which removes scenarios in which the realization of the demand uncertainty is ``large" in absolute value for all times, but includes scenarios in which the realization of the demand uncertainty is ``large" only for a subset of times. When solving \eqref{eq.RSPP}, smaller uncertainty sets allow us to reduce our conservatism, and hence improve our performance, assuming the uncertainty set contains the (unknown) true demand. This will become evident in the next section, which will test the prescribed algorithm in a case study.

To conclude, the mixed uncertainty set \eqref{eq.MixedNorm} considers more realistic scenarios, in which the demand does not simultaneously peak at all time steps. In practice, if the nominal demand $(P_0(t),H_0(t))_{t=1}^T$ is achieved from a load prediction algorithm, e.g. one appearing in \cite{Gross1987,Akay2007,Yu2017,Mirasgedis2006}, the parameters $\Delta P(t),\Delta H(t), \delta_{P,t},\delta_{H,t}, \mu_1$ are fitted by looking at the deviation of the algorithm. Some prediction algorithms report their expected standard deviation, while for others, we can look at the accuracy of the algorithm in the past. It might seem as if we return to the previous problem, where parameter inaccuracy jeopardized the performance of the algorithm, there are evidence that the robust algorithms are much less vulnerable to such parameter inaccuracies \cite{Bertsimas2004}.

Before moving on to the case studies illustrating the performance of the algorithms, we wish to elaborate a bit more on the complexity of the algorithms displayed above in terms of the discretization of the turbine:

\begin{rem}
Corollary \ref{cor.Linfty} and Theorems \ref{thm.L1Linfty}, \ref{thm.AdditiveApproximation} and \ref{thm.MultiApprox} prescribe complexity bounds on the algorithms presented in this section. The complexity bound is stated in terms of the number of edge $|\E|$ in the graph, which is the custom for graph-based algorithms in computer science \cite{Cormen2009}. However, we would like to understand it in terms of the discretization of the turbine.

The number of vertices $|\V|$ in the graph is equal to $T|\X|$, where $T$ is the horizon of the problem and $\X$ is the discrete state-space of the turbine. The relationship between the number of vertices and the number of edges is a bit more complex. Generally, we know that $|\V| \le |\E| \le \|\V\|^2$, as any vertex is connected to at least one other vertex, but no more than $\|\V\|$ other vertices. A more precise relationship can be achieved by considering the dynamics of the discretized model. Namely, the number of edges is smaller than $T\sum_{x\in \X} |\U(x)|$, and both are roughly equal if $T$ is much larger than most transition lengths $c(x,u)$. If we assume that the number of control actions in each state is bounded between $\U_{\rm min}$ and $\U_{\rm max}$, then $TU_{\rm min}|\X|\le |\E| \le TU_{\rm max}|\X|$. In that case, linear-time and quadratic-time algorithms in $|\E|$ are also linear-time and quadratic-time algorithms in $|\X|$, respectively.
\end{rem}
\section{Case Studies} \label{sec.CaseStudy}
\subsection{Modeling and Pricing}
We demonstrate the benefit of the presented algorithm in the economic dispatch of an MGT for CHP operation, whose cost functions are inspired by a discretized version of the Capstone C65 turbine \cite{C65Capstone}. The engine unit  consists of a single stage centrifugal compressor, a can-type combustor, a single stage turbine, a recuperator and a separate heat recovery unit. In order to accommodate the changing ratio between power and heat generation demand, the recuperator is equipped with a controllable valve which alters the amount of exhaust gasses bypassing the heat exchanger. The cycle's schematic is presented in Fig. \ref{fig.SchematicCycle}.

\begin{figure} [H]
    \centering
    \includegraphics[width = 0.45\textwidth]{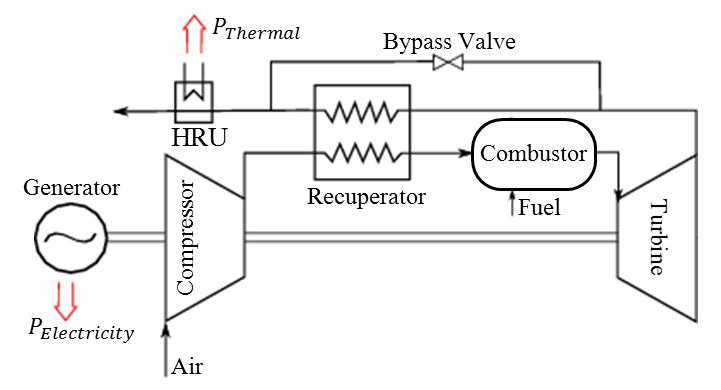}
    \caption{Engine schematic cycle.}
    \label{fig.SchematicCycle}
\end{figure}

\begin{figure*}[t]
    \centering
    \subfigure[Fuel mass flow.] {\scalebox{.39}{\includegraphics{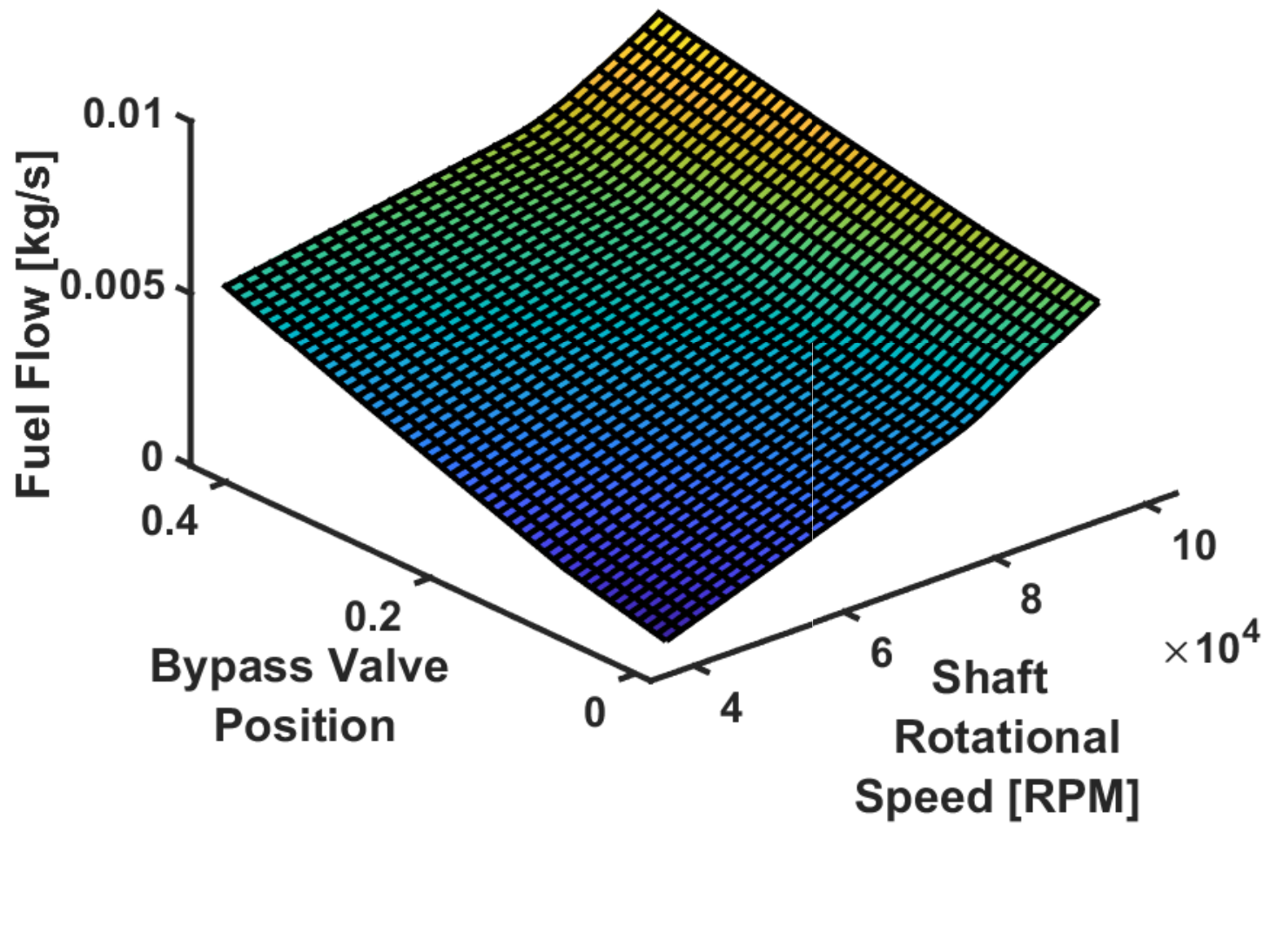}}} \hspace{5pt}
    \subfigure[Electricity output.] {\scalebox{.39}{\includegraphics{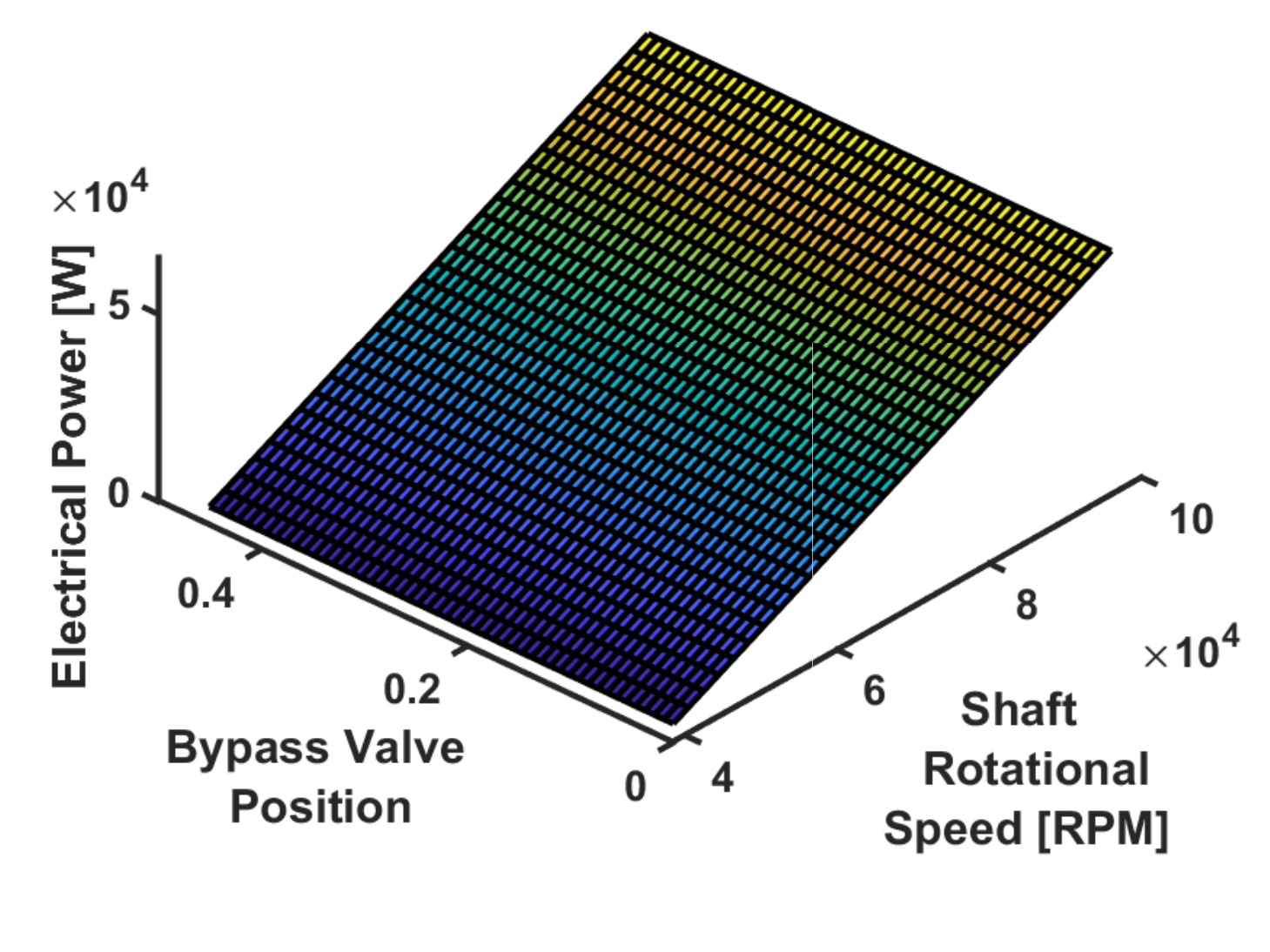}}}\hspace{5pt}
    \subfigure[Heat output.] {\scalebox{.39}{\includegraphics{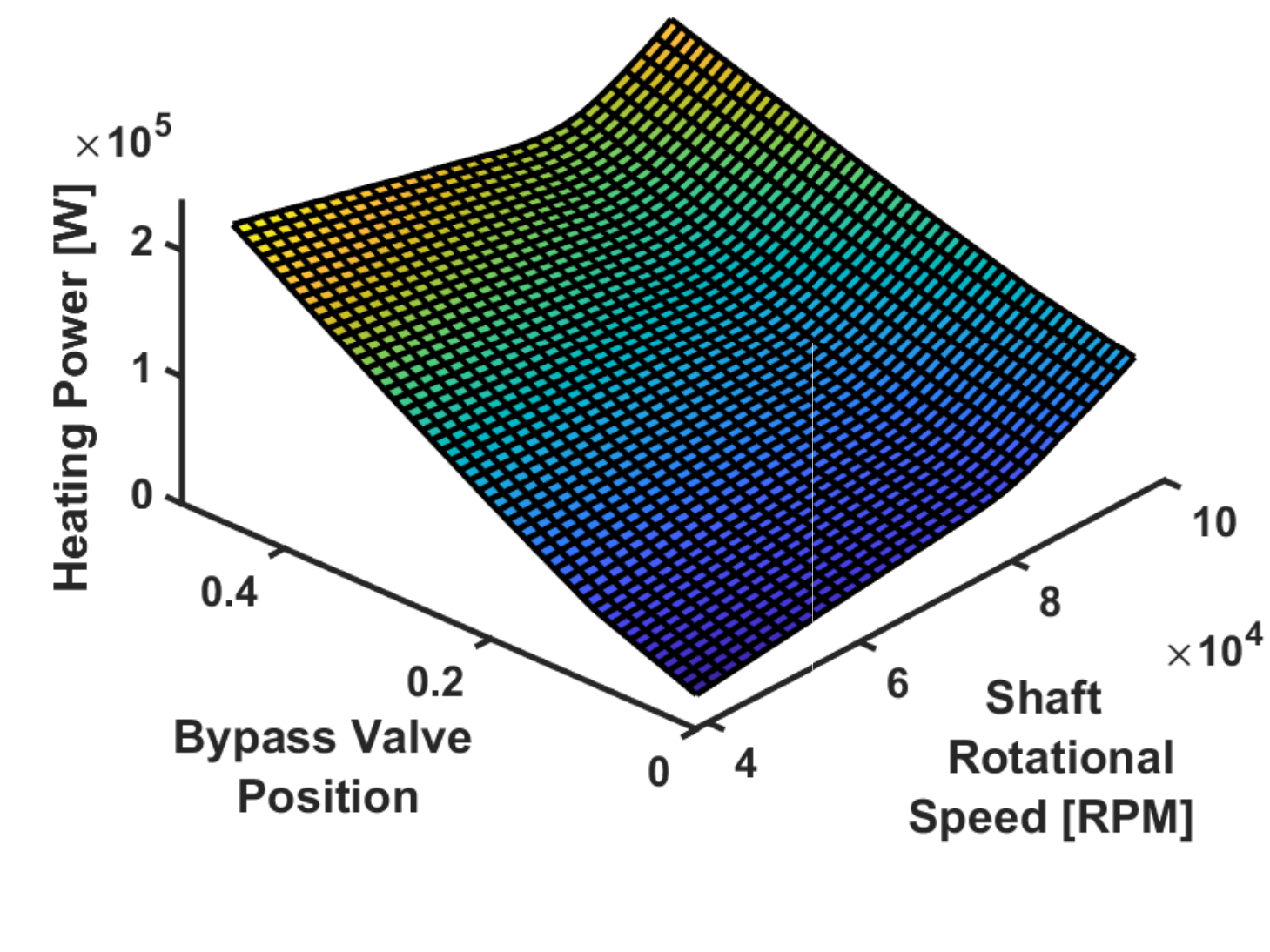}}}
    \caption{Solution grid over the states of the gas turbine model.}
    \label{fig.ElectricalPowerDomains}
\end{figure*}

In order to optimize the MGT cycle during its operation, two input parameters (shaft speed and recuperator bypass valve position) are selected and simulated to yield a number of solution states (electrical power and heat output that prescribes fuel mass flow). The discrete state space consists of $1501$ states ($1500$ active states and one `off' state), corresponding to $30$ different engine shaft speeds ($38.4-96$ krpm) and $50$ bypass valve positions ($0-45\%$). The thermodynamic performance of the MGT is characterized in electrical power and heat output domains ranging between $5-65$ kWel and $27-216$ kW respectively; see Fig. \ref{fig.ElectricalPowerDomains}(a)–(c). The lowest heat to power output ratio varies between $1.7-3.3$ for $0-45\%$ bypassing conditions. 

We adopt the edge weight scheme appearing in \cite{Rist2017} and determine the cost of an edge in the graph $\G$ as follows: for an edge between a state $x(t_1)$ and a state $x(t_2)$, the power generation is defined as the average power $\frac{x(t_1) + x(t_2)}{2}$ times the transition time $t_2-t_1$. The transition time is defined as $15$ seconds if we do not increase the engine rpm, and as $30$ seconds if we increase it by one step. The heat generation, fuel consumption and demand on the edge are defined similarly. The cost of the edge is now defined as the sum of the fuel cost, plus the cost of buying power and heat from the utility, while ensuring the power- and heat-balance equations are satisfied. In addition, there is a cost associated with the start-up and shut-down of the unit. 

The cost of a gas turbine engine is roughly $\$75,000$. Thus, assuming a low cycle fatigue life of 10,000 cycles \cite{Majumdar1975}, we estimate the cost of shutdown and startup as $\$3.75$ each. We assume the MGT is offline throughout start-up (lasting $6$ minutes) and shut-down (lasting $3$ minutes) sequences, so all power and heat must be purchased from the utility.

The dispatch problem is considered for a residential building of multiple apartments. The demand profile of each apartment stems from the data published by the U.S. Department of Energy for the entire year of 2004 \cite{EnergyBuildings1,EnergyBuildings2,EnergyBuildings3}. Within this database, we only consider the residential buildings internally specified as ``Residential High."\footnote{The full name of the file is {\tt\small USA\_NY\_New.York-Central.Park\\.725033\_TMY3\_HIGH.csv}} In order to decide how many apartments benefit from the same MGT, the demand is scaled such that $95$\% of the time, the turbine's electrical capacity is $80$\% of the consumer demand. This roughly corresponds to the needs of 9.3 apartments for our 65kWe turbine capacity.

For the cost of energy supplied from the utility, the price of electricity was determined according to data from PSEG Long Island New York \cite{PSEGLINY}, similar to \cite{Rist2017}. In particular, the price of electricity is different between peak hours\footnote{These are the hours between 10AM and 8PM.} and off-peak hours, as well as between winter days\footnote{These are days between October 1$^{\rm st}$ and May 31$^{\rm st}$.} and summer days.\footnote{These are days between June 1$^{\rm st}$ and September 30$^{\rm st}$.} See \cite[Table 2,3]{Rist2017} for residential buildings for more information. In addition to electricity, the energy source for heat/chill can also be obtained from the utility to achieve energy balance between the demand and supply. As most consumers use a boiler to satisfy their heat demand, we model the heating cost with the price of natural gas - as shown in Eq. 24 in Ref. \cite{Rist2017}. As for the cost of energy supplied by the MGT, the natural gas is the only consumable. The price of natural gas was taken to be $\$18.42$ per thousand cubic feet\footnote{Or equivalently, about $95$ cents per kilogram.}, which was the residential price of natural gas in August 2020 \cite{PriceOfGas}. If the MGT produces excess electricity beyond that of the local demand, the output is sold to the utility at the same tariff rate as the retail cost.

\textcolor{black}{In our previous work \cite{Rist2017}, it was shown that the solution to the ED problem under deterministic and  known demand profiles leads to four fundamental modes of  operation for the MGT:
%
%To summarize, the MGT can operate in response to
%under four essential, deterministic demand profiles identified in previous research effort \cite{Rist2017}: 
electricity, heat, maintenance-cost and profit driven. Electricity and heat driven modes are dependent on the customer demand profiles in which peak electricity and peak heat requests vary through yearly seasons and day hours. Maintenance-cost driven behaviour is geared towards minimizing losses associated with low cycle fatigue cost for each shutdown and startup, along with fuel burned during these sequences. This operational state is manifested by running the MGT at a minimal electricity and heat production level (typically during the off-peak tariff period) despite the fulfillment of the demand from the utility appears to be more economical. Finally, in the profit driven state, the MGT operates at an electricity production level significantly beyond the consumer demand, selling all the excess electricity to the grid.}

\subsection{Algorithms and Running Time}
We treat the year 2004 as real time for  assessing the performance of each algorithm. The dispatch problem is considered for a horizon of 24 hours, dividing it to $T=5760$ intervals, each 15 seconds long. Although we have the ground-truth demand for each day, it is not available when solving the economic dispatch problem in practice, since we can not predict the ``future" accurately. Instead, we must use an estimate of the demand. For a given day, we use the ground-truth demand data from the previous two weeks (which are indeed available in practice while solving the ED problem) to compute an estimate of the demand for that day. For each time of day $t$, we compute the sample mean $\mu_P(t),\mu_H(t)$ and the standard deviation $\sigma_P(t),\sigma_H(t)$ of both power and heat demand from the preceding two weeks.

The performance of three algorithms are compared. Firstly, the ED algorithm from \cite{Rist2017} that does not account for demand uncertainty is applied to the forecasted mean demand $(\mu_P(t),\mu_H(t))$ - we term this as the \emph{nominal} algorithm. Moreover, the two robust ED algorithms presented in this work are considered. Their uncertainty sets use the standard deviation $\sigma_P(t),\sigma_H(t)$ of the power and heat demands, in addition to the forecasted mean demand. 

The first uncertainty set we choose is of the form \eqref{eq.Linfty} with $P_0(t)=\mu_P(t),H_0(t)=\mu_H(t)$ and $\Delta P(t) = \alpha_{\LL_\infty} \sigma_P(T), \Delta H(t) = \alpha_{\LL_\infty} \sigma_H(t)$ for some parameter $\alpha_{\LL_\infty}>0$. The tuning parameter $\alpha_{\LL_\infty}$ represents the trade-off between conservatism and accuracy. If $\alpha_{\LL_\infty}$ grows larger, the probability that the ground-truth demand is inside the uncertainty set becomes bigger. However, this also implies the algorithm considers greater likelihood for outlier events associated with worst-case demand profiles; hence the solution becomes more conservative. In this study, $\alpha_{\LL_\infty}$ is selected to be 0.13.
The second uncertainty set we choose is of the form \eqref{eq.MixedNorm} where $P_0(t)=\mu_P(t),H_0(t)=\mu_H(t)$, $\Delta P(t) = \alpha_{\rm Mixed,1} \sigma_P(T), \Delta H(t) = \alpha_{\rm Mixed,1} \sigma_H(t)$, $\delta_{P,t} = \frac{1}{\sigma_P(t)}$, $\delta_{H,t} = \frac{1}{\sigma_H(t)}$, and $\mu_1 = \alpha_{\rm Mixed,2}$. The parameters $\alpha_{\rm Mixed,1},\alpha_{\rm Mixed,2}$ determine the size of the uncertainty set and are tunable. Here, we chose $\alpha_{\rm Mixed,1} = 0.03$ and $\alpha_{\rm Mixed,2} = 40$.

All the algorithms were computed on a Dell Latitude 7400 computer with an Intel Core i5-8365U processor. The nominal algorithm utilizes the shortest path formulation presented in \cite{Rist2017}. For the first robust ED algorithm choice, Corollary \ref{cor.Linfty} shows that a single application of the shortest path algorithm suffices as well. Both the nominal algorithm and the first robust algorithm use the same underlying combinatorial solution provided by the MATLAB internal shortest path solver, meaning they have nearly identical running times.  For a total of 1500 discretization level combinations, building the graph and finding the shortest path takes about 2 minutes. 
For the second robust ED algorithm choice, we must apply Algorithm  \ref{alg.MixedL1L2}. As the number of edges in the underlying graph is in the millions, the application of Algorithm \ref{alg.MixedL1L2} would require running the shortest path algorithm on roughly 10 million different graphs. We instead use the approximate version described in Theorem \ref{thm.AdditiveApproximation}, where $\epsilon$ is chosen such that exactly $N=30$ applications of the shortest path problems are performed. Even then, the runtine is significantly longer, about 7 minutes. However, it is  still considered fast enough to be applicable in real-world systems. (Note that these runtimes can vary for different choices of the turbine discretizations.)

In addition to these three cases that forecast the demand, in order to contrast the performance of the algorithms with the global optimum, the shortest-path algorithm is applied to the ground-truth demand, which cannot be used in practice, as this quantity is unknown at the time of scheduling. Since this case produces the best possible schedule, we present it as the "benchmark" in all solutions.

\subsection{Schedules and Associated Costs for Residential Buildings}

In order to demonstrate the performance of the algorithms, a few exemplary days were analyzed considering their known two week demand histories: one winter day (February 5$^{\rm th}$), one spring day 
(March 24$^{\rm th}$), one summer day (June 28$^{\rm th}$), and one autumn day (September 19$^{\rm th}$). We note that the winter electricity tariff is used in the winter and spring days, whereas the summer electricity tariff is used in the summer and autumn days. For each day, the two robust algorithms (with demand uncertainty) and the nominal algorithm (without demand uncertainty) are evaluated based on the forecasted demand, and compared to the benchmark schedule stemming from known ground truth demand. Figures \ref{fig.Case1}-\ref{fig.Case4} present the resulting MGT electricity and heat production schedules of the 3 algorithms and contrasts it with the benchmark case utilizing ground-truth demand, which is also charted separately. The blue bands around the power and heat schedules indicate the forecasted demands with the standard deviations utilized in the uncertainty sets of the robust algorithms.    In order to clarify the actual impact on the engine control parameters, spool speed and bypass valve position are separately charted for each algorithm. The schedules are computed for $T=5760$ intervals, each 15 seconds long.
 
In the winter day of February 5$^{\rm th}$, Fig. \ref{fig.Case1}, it can be seen that all four algorithms keep the MGT spool speed constant at its minimum operational value at almost all times, yielding the smallest amount of local electricity production, consistently below the power demand, which is predominantly met by purchase from the utility. However, the differences in schedules arise from changes in bypass valve schedule and the associated MGT heat production. \textcolor{black}{Thus, the solution seems to be heat demand driven}. The first robust algorithm gives the closest results to the benchmark case, followed by the second robust algorithm, and the worst case is the nominal algorithm absent of uncertainty.

Similarly, in the spring day of March 24$^{\rm th}$, Fig. \ref{fig.Case3}, all four algorithms keep the MGT spool speed constant at its minimum operational value
at almost all times, significantly below the power demand, which is met by purchase from the utility. The second robust algorithm produces the same schedule as the nominal algorithm, but the first robust algorithm and the benchmark induce different schedules by choosing different trajectories for the bypass valve. \textcolor{black}{This suggests that the solution is heat driven also in this scenario.}
The first robust algorithms slightly outperforms the others, although the schedules produced by all three algorithms have similar associated costs.

In the summer day of June 28$^{\rm th}$, Fig. \ref{fig.Case6}, the first robust algorithm and the nominal algorithm produce identical schedules, and the second robust algorithm produces almost the same schedule as the benchmark. All algorithms decide to start the day operational. In the morning hours of the day (0-10AM), the MGT spool speed is near its minimum level and the electricity production is below the local demand. Moreover, the bypass valve is set to 0\%,  \textcolor{black}{manifesting maintenance-cost driven behaviour}. At around 10AM, the MGT spool speed increases to about 78krpm in order to satisfy the local demand. However, the electricity output of the turbine is not maximized, as the MGT does not reach its highest rpm. This appears to be an electricity demand driven operation, which lasts until around 19:00. From 20:00 until midnight, the first robust and the nominal algorithms decide to turn the turbine off, where as the second robust algorithm decides to keep it operating close to the minimum capacity. In fact, the second robust algorithm achieves a nearly identical schedule to the benchmark, with only a small change around 8-9PM, and a difference in cost of less than $1$ cent.

In the fall day of September 19$^{\rm th}$, see in Fig. \ref{fig.Case4}, according to the benchmark case, the optimal schedule is to shut down the MGT for the entire day. \textcolor{black}{This corresponds to a maintenance-cost driven solution, as it implies that the possible savings that could be achieved by turning on the turbine are negated by the low cycle fatigue cost associated with each startup and shutdown.} Only the second robust algorithm manages to replicate this behaviour. Both the first robust and the nominal algorithms turn on the turbine, and achieve an identical schedule, with a difference in cost of $\$6.4$.

\begin{figure} [!ht] 
    \centering
    \includegraphics[width = 0.45\textwidth]{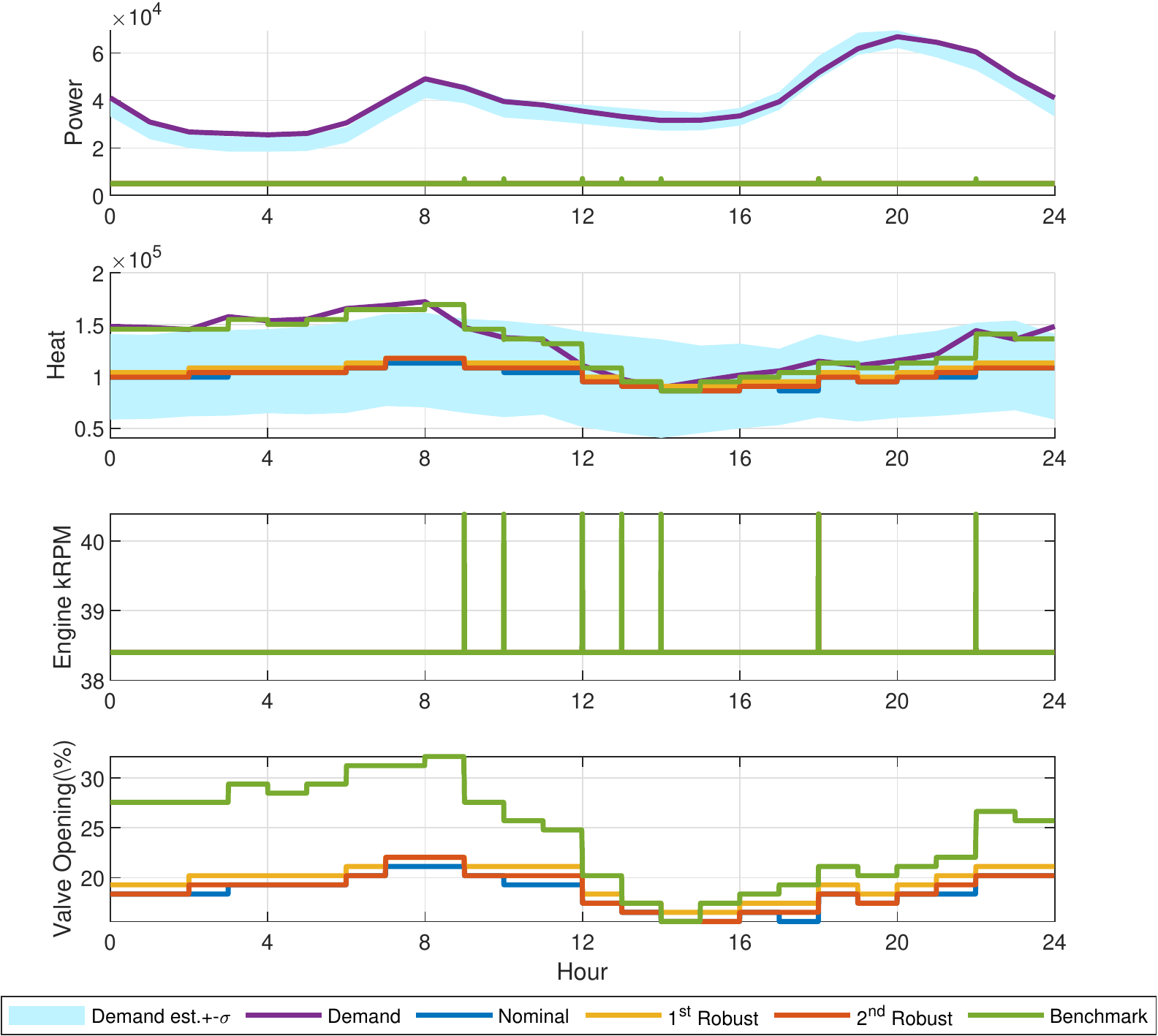}
    \caption{Schedules produced by the algorithms for February 5$^{\rm th}$ (winter).}
    \label{fig.Case1}
\end{figure}

%\begin{figure} [!ht] 
%    \centering
%    \includegraphics[width = 0.45\textwidth]{CaseStudy/ScheduleFigure_50x30_2.eps}
%    \caption{The schedules produced by the algorithms for the second winter day, November 8$^{\rm th}$. The horizontal axis represents the hour of day. The second robust algorithm produces the same schedule as the nominal algorithm.}
%    \label{fig.Case2}
%\end{figure}

\begin{figure} [!ht] 
    \centering
    \includegraphics[width = 0.45\textwidth]{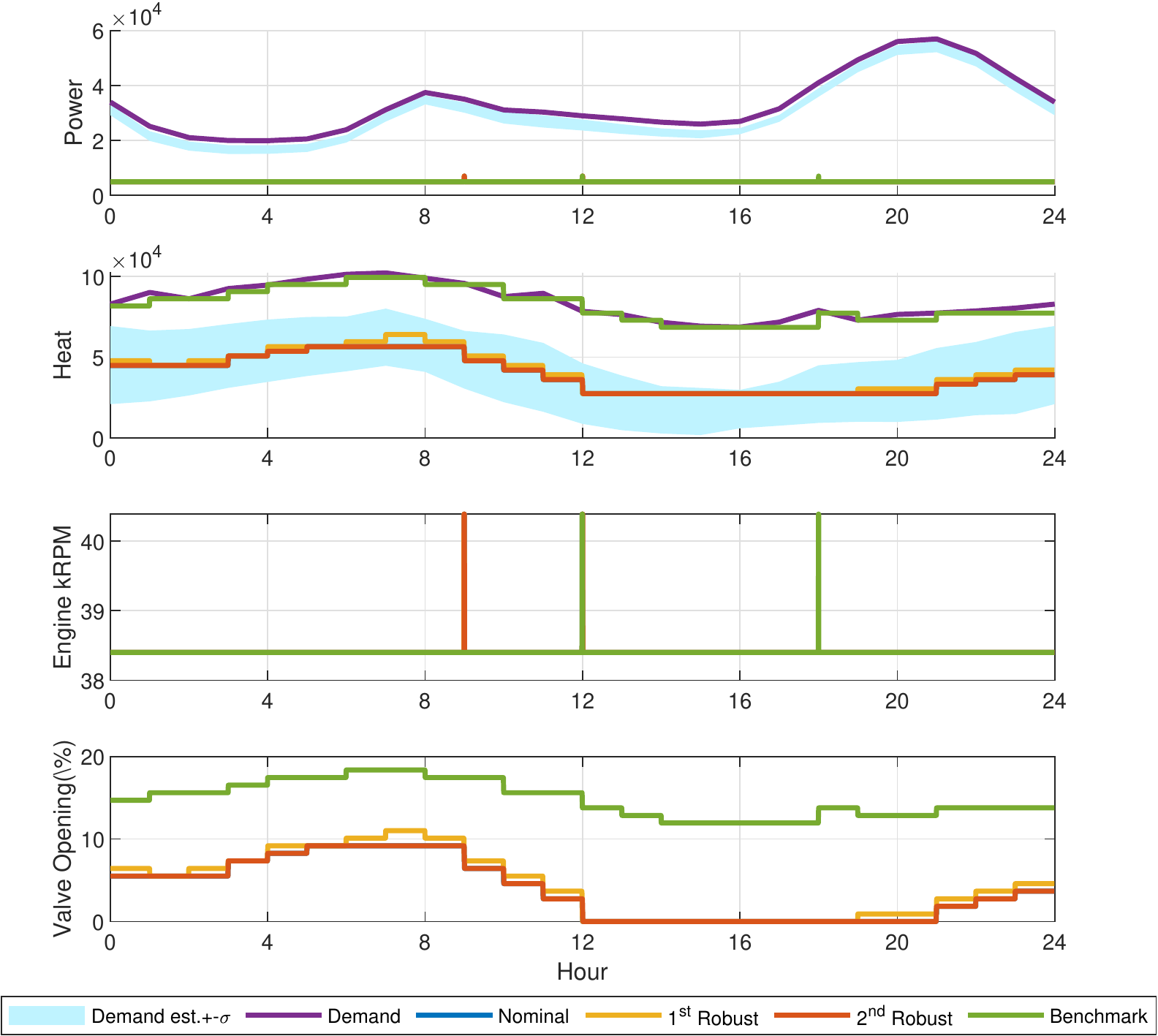}
    \caption{The schedules produced by the algorithms for the spring day, March 24$^{\rm th}$. The second robust algorithm produces the same schedule as the nominal algorithm.}
    \label{fig.Case3}
\end{figure}

\begin{figure} [!ht] 
    \centering
    \includegraphics[width = 0.45\textwidth]{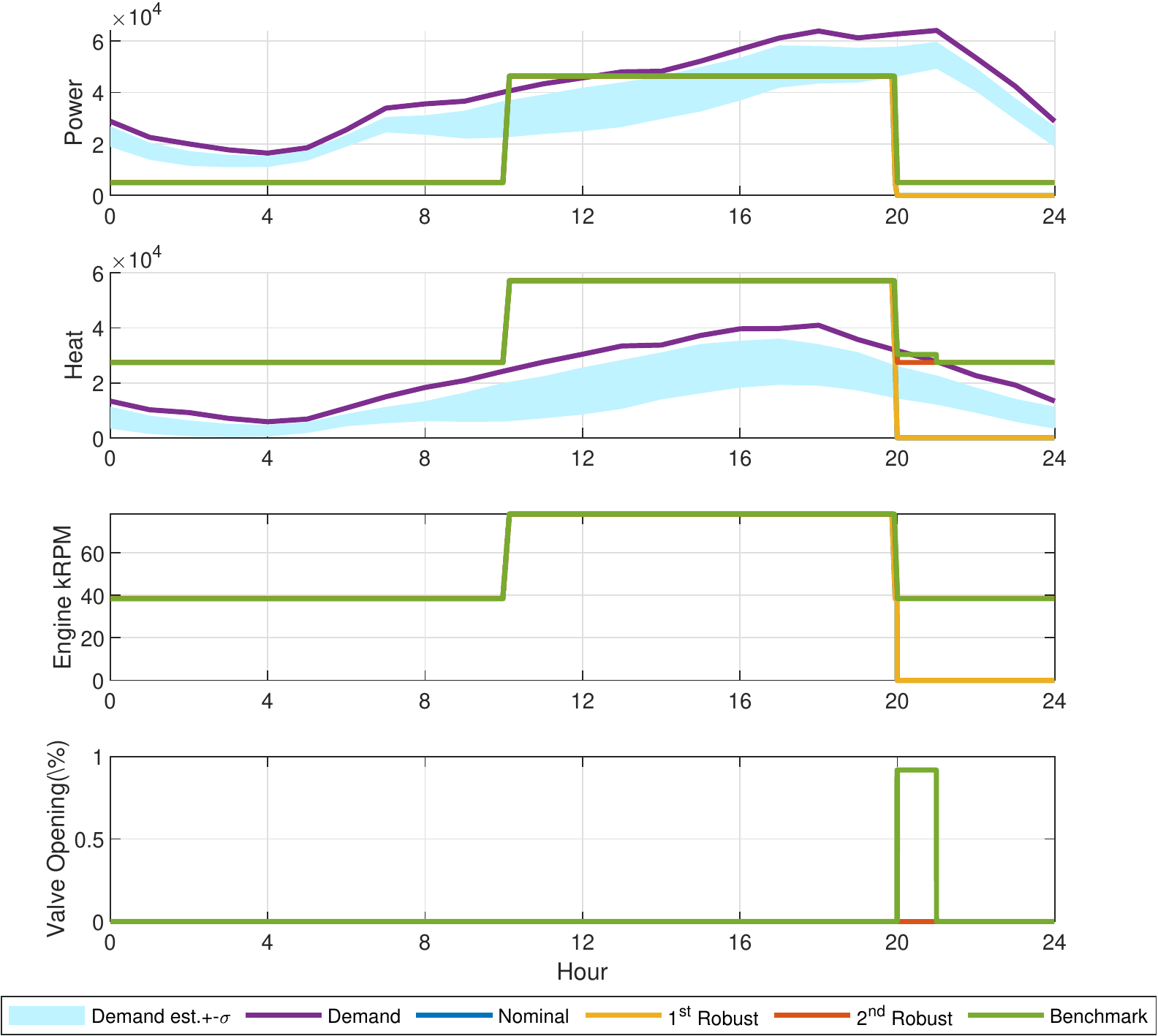}
    \caption{The schedules produced by the algorithms for the summer day, June 28$^{\rm th}$. The first robust algorithm produces the same schedule as the nominal algorithm, and the second robust algorithm produces a similar schedule to the benchmark.}
    \label{fig.Case6}
\end{figure}

%\begin{figure} [!ht] 
%    \centering
%    \includegraphics[width = 0.45\textwidth]{CaseStudy/ScheduleFigure_50x30_5.eps}
%    \caption{The schedules produced by the algorithms for the first summer day, June 9$^{\rm th}$. The horizontal axis represents the hour of day. The first robust algorithm produces the same schedule as the nominal algorithm.}
%    \label{fig.Case5}
%\end{figure}

\begin{figure} [!ht] 
    \centering
    \includegraphics[width = 0.45\textwidth]{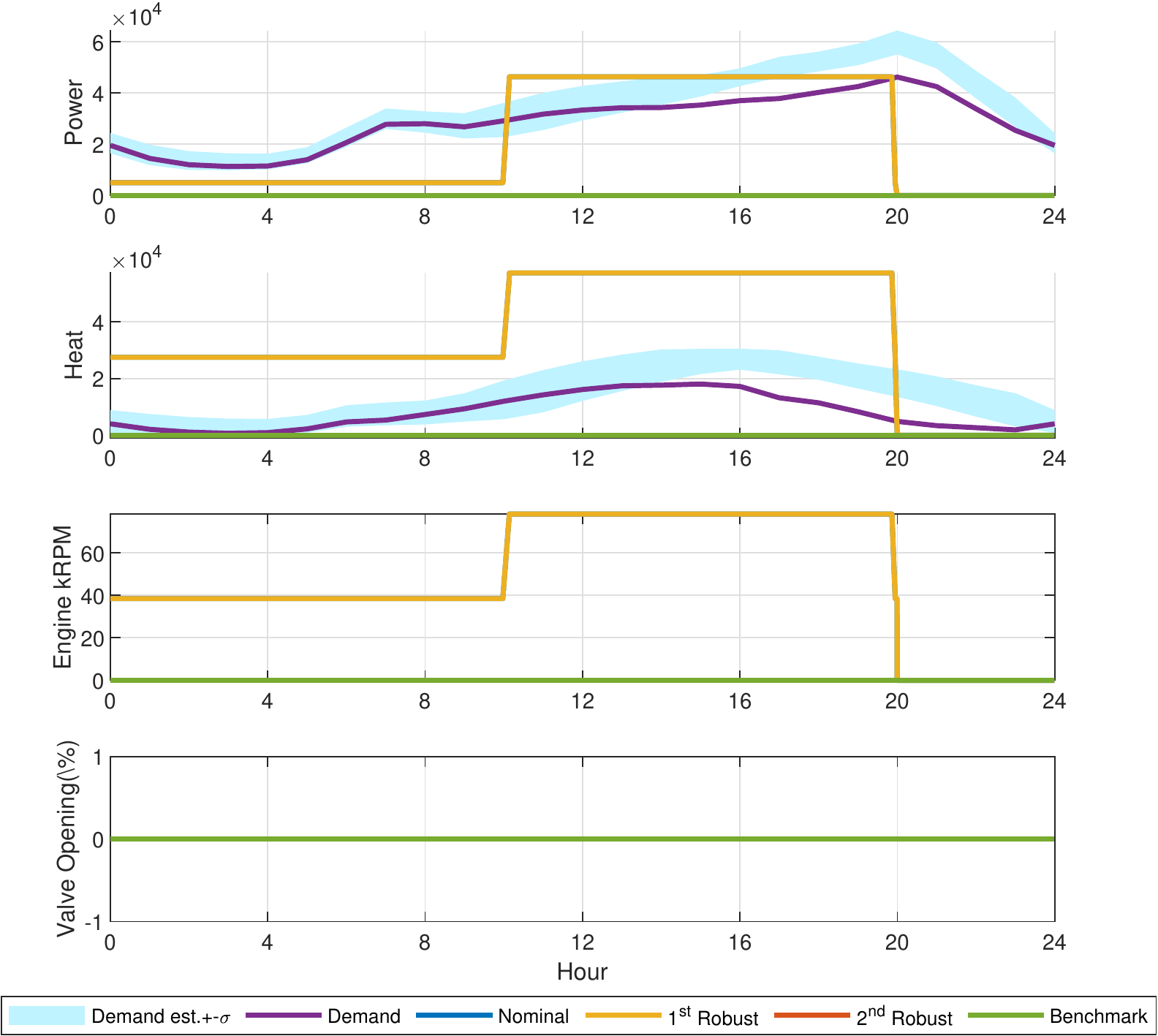}
    \caption{The schedules produced by the algorithms for the autumn day, September 19$^{\rm th}$. The first robust algorithm produces the same schedule as the nominal algorithm, and the second robust algorithm produces the same result as the benchmark.}
    \label{fig.Case4}
\end{figure}

Table \ref{table.SchedCost} summarizes the daily scheduling cost of all cases, where the most favorable forecasting solution of that day is highlighted by italics.  In each of the days, both robust algorithms perform at least as well as the nominal algorithm. Moreover, the first robust algorithm outperformed the nominal algorithm for all days with winter electricity tariff, with savings up to $\$0.89$ per day. The second robust algorithm outperformed the nominal algorithm for all days with summer electricity tariff, with savings up to $\$6.37$ per day. In fact, the second robust algorithm has exactly the same scheduling cost as the benchmark case in these days, meaning it successfully finds the global minimum.

We can consider an alternative performance metric for the robust algorithms. It is clear that the benchmark cost is the optimum over all possible schedules. The nominal algorithm is an uncertainty-agnostic algorithm which represents the baseline from which we begin, in an effort to reduce the cost. The margin in cost between the nominal algorithm and the benchmark represents the potential benefit that any robust algorithm can offer. In table \ref{table.SchedCost}, indicated in parentheses, we calculate the reduction of excess cost as a percentage of this margin, such that $100\%$ and $0\%$ reduction, implies that the robust algorithm performs identically to the benchmark and nominal cases respectively. For the four days considered, the first robust algorithm has an average reduction in excess cost of about $4\%$, and the second robust algorithm has an average reduction in excess cost of about $51\%$. This is expected considering the discussion in Section \ref{subsec.L1LinftyBalls}, as the first robust algorithm has an uncertainty set of the form \eqref{eq.Linfty} and the second robust algorithm has an uncertainty set of the form \eqref{eq.MixedNorm}.

\begin{table*}[h!]
\begin{center}
\begin{tabular}{ |c|c|c|c|c| } 
\hline
\multirow{2}{*}{\makecell{\bf Schedule Cost in \$ \\\bf (Reduction in Excess Cost in \%)}} & {\bf Winter} & {\bf Spring} &{\bf Summer}  & {\bf Autumn}\\ 
  & Feb. 5$^{\rm th}$ &
    Mar. 24$^{\rm th}$ & 
    Jun. 28$^{\rm th}$ & Sep. 19$^{\rm th}$ \\ 
  \hline
  {\bf Benchmark Case}& 293.02 & 196.86 & 188.83 & 126.48 \\ \hline
  {\bf Nominal Algorithm} & \makecell{299.39 } & \makecell{202.30} & \makecell{191.35}  & \makecell{133.32}\\ \hline
  {\bf First Robust Algorithm} & \makecell{\emph{298.48} \\ \emph{(14.29\%)}} & \makecell{\emph{202.16} \\ \emph{(2.57\%)}} & \makecell{191.35 \\ (0.00\%)} & \makecell{133.32 \\ (0.00\%)}\\ \hline
  {\bf Second Robust Algorithm} & \makecell{299.16\\(3.61\%)} & \makecell{202.30 \\ (0.00\%)} & \makecell{\emph{188.83} \\ \emph{(100.00\%)}} & \makecell{\emph{126.48} \\ \emph{(100.00\%)}}\\
  \hline
\end{tabular}
\end{center}
\caption{Schedule costs for the benchmark algorithm, the nominal algorithm, and the two robust algorithms. The reduction of excess costs for the robust algorithms compared to the nominal algorithm is displayed in parentheses. The costs associated with the best-performing algorithm in each day are highlighted by italics.}
\label{table.SchedCost}
\end{table*}

\section{Conclusion}
We considered the economic dispatch problem with uncertain demand for a single micro gas turbine, providing combined heat and power, coupled with utility. %We assumed that the dynamics of the turbine are given by a discrete state-space model, which can be achieved by discretizing the operational regions of a gas turbine map, characterizing the fuel consumption as a function of spool speed and bypass valve position, that prescribes the local electricity and heat production. 
We considered the case in which the demand is assumed to be contained in a given uncertainty set\textcolor{black}{, and showed an equivalence between the economic dispatch problem and the robust shortest-path problem}. Two different models of an uncertainty set were proposed: one including time-dependent confidence intervals with no further assumption, and another coupling with a budgeting assumption throughout the time horizon. \textcolor{black}{Both algorithms relied on adaptations of the classical shortest-path problem, and we presented proofs for their correctness and analyses of their time complexity.} %In the first case, we reduced the problem to a shortest path problem, resulting in a linear-time algorithm for finding the optimal solution. In the second case, we constructed an algorithm finding the optimal solution in quadratic time, and explained how to modify it to give an approximate optimal solution in linear time. 
Both proposed algorithms were demonstrated in a case study, in which we examined their performance under realistic demand framework and tariffs of a residential unit. \textcolor{black}{Our results indicate that the robust algorithms proposed in this manuscript with forecasted demand and uncertainty sets outperform the nominal, non-robust algorithm with forecasted demand. More precisely, when the benchmark algorithm displays heat-driven behaviour, the $\mathcal{L}_\infty$-norm robust algorithm outperforms the nominal algorithm. Moreover, in electricity-driven and maintenance-cost driven settings, the mixed-norm robust algorithm outperforms the nominal algorithm, and actually} reaches the globally optimal schedule stemming from a fully known heat and power demand. 

This is a first step toward a robust integration of micro gas turbines with complex models and restrictions into a micro-grid setting, which cannot deterministically predict the future demand.
The presented methods can also be applied to the case of arrays of multiple gas turbines by defining $x(t)$ as a tuple including the states of all the turbines in the array. Unfortunately, this methods scales exponentially with the number of turbines, so it can only be applied to modest size arrays.
Future work can try to improve the performance of the algorithms on large turbine arrays by either partitioning the corresponding robust shortest path problem to multiple smaller problems, or by using dual-gradient methods, which will apply the robust algorithms described herein as intermediate steps when calculating the gradient. \textcolor{black}{Another possible avenue for future research is the development of more complex scheduling strategies which consider updates in the demand uncertainty over time. The demand uncertainty can change when a part of the unknown demand in revealed. Such methods can use a rolling-horizon, receding-horizon or a model-predictive control approach, all relying on iterative optimization over time, and will therefore rely on the robust-shortest path algorithms developed in this paper.}

\section{Acknowledgments}
The authors acknowledge the financial support of Minerva Research Center for Micro Turbine Powered Energy Systems (Max Planck Society Contract AZ5746940764),  and Startups in Energy Program of Israeli Ministry of Energy (Contract 20180805). 
The first author thanks Dean Leitersdorf for helpful discussions.

\bibliographystyle{elsarticle-num} 
\bibliography{References}

\begin{thebibliography}{10}
\expandafter\ifx\csname url\endcsname\relax
  \def\url#1{\texttt{#1}}\fi
\expandafter\ifx\csname urlprefix\endcsname\relax\def\urlprefix{URL }\fi
\expandafter\ifx\csname href\endcsname\relax
  \def\href#1#2{#2} \def\path#1{#1}\fi

\bibitem{Pilavachi2002}
P.~Pilavachi, Mini-and micro-gas turbines for combined heat and power, Applied
  thermal engineering 22~(18) (2002) 2003--2014.

\bibitem{Gu2014}
W.~Gu, Z.~Wu, R.~Bo, W.~Liu, G.~Zhou, W.~Chen, Z.~Wu, Modeling, planning and
  optimal energy management of combined cooling, heating and power microgrid: A
  review, International Journal of Electrical Power \& Energy Systems 54 (2014)
  26--37.

\bibitem{Liu2014}
M.~Liu, Y.~Shi, F.~Fang, Combined cooling, heating and power systems: A survey,
  Renewable and Sustainable Energy Reviews 35 (2014) 1--22.

\bibitem{Mongibello2015}
L.~Mongibello, N.~Bianco, M.~Caliano, G.~Graditi, Influence of heat dumping on
  the operation of residential micro-chp systems, Applied energy 160 (2015)
  206--220.

\bibitem{Pantaleo2013}
A.~Pantaleo, S.~Camporeale, N.~Shah, Thermo-economic assessment of externally
  fired micro-gas turbine fired by natural gas and biomass: Applications in
  italy, Energy Conversion and management 75 (2013) 202--213.

\bibitem{Rist2017}
J.~F. Rist, M.~F. Dias, M.~Palman, D.~Zelazo, B.~Cukurel, Economic dispatch of
  a single micro-gas turbine under chp operation, Applied energy 200 (2017)
  1--18.

\bibitem{C65Capstone}
Technical and service manual. {Capstone} microturbine, model {C65},
  \url{https://www.capstoneturbine.com/} [Accessed 2020-08-08].

\bibitem{AMTNikeManual}
{AMT NIKE} manual and engine log, \url{http://www.amtjets.com/Nike.php}
  [Accessed: 2020-08-08].

\bibitem{Lee1993}
F.~N. Lee, A.~M. Breipohl, Reserve constrained economic dispatch with
  prohibited operating zones, IEEE transactions on power systems 8~(1) (1993)
  246--254.

\bibitem{Papageorgiou2007}
L.~G. Papageorgiou, E.~S. Fraga, A mixed integer quadratic programming
  formulation for the economic dispatch of generators with prohibited operating
  zones, Electric power systems research 77~(10) (2007) 1292--1296.

\bibitem{FederalRegister}
U.~S.~A. Federal Aviation Administration. Department~of Transportation, Rules
  and regulations, vol. 71, No. 188, Sept. 28, 2006, pp. 56864-56866.
  https://www.govinfo.gov/content/pkg/FR-2006-09-28/pdf/FR-2006-09-28.pdf
  [Accessed: 2020-08-08].

\bibitem{CapeceV}
V.~R. Capece, Y.~M. EL-Aini, Stall flutter prediction techniques for fan and
  compressor blades, Journal of Propulsion and Power 12~(4) (1996).

\bibitem{Bendiksen1988}
O.~Bendiksen, Recent developments in flutter suppression techniques for
  turbomachinery rotors, Journal of Propulsion and Power 4~(2) (1988) 164--171.

\bibitem{LieuwenT}
T.~C. Lieuwen, V.~Yang, Gas Turbine Emissions, Cambridge University Press,
  2013.

\bibitem{KellerJ}
J.~J. Keller, Thermoacoustic oscillations in combustion chambers of gas
  turbines, AIAA Journal 33~(12) (1995).

\bibitem{TimothyC}
T.~C. Lieuwen, V.~Yang, Combustion Instabilities In Gas Turbine Engines:
  Operational Experience, Fundamental Mechanisms and Modeling, American
  Institute of Aeronautics and Astronautics, 2006.

\bibitem{Happ1977}
H.~H. {Happ}, Optimal power dispatch - a comprehensive survey, IEEE
  Transactions on Power Apparatus and Systems 96~(3) (1977) 841--854.

\bibitem{Wood2013}
A.~J. Wood, B.~F. Wollenberg, G.~B. Shebl{\'e}, Power generation, operation,
  and control, John Wiley \& Sons, 2013.

\bibitem{Bertsekas1983}
D.~Bertsekas, G.~Lauer, N.~Sandell, T.~Posbergh, Optimal short-term scheduling
  of large-scale power systems, IEEE Transactions on Automatic Control 28~(1)
  (1983) 1--11.

\bibitem{Zelazo2012}
D.~Zelazo, R.~Dai, M.~Mesbahi, An energy management system for off-grid power
  systems, Energy Systems 3~(2) (2012) 153--179.

\bibitem{Binetti2014}
G.~Binetti, A.~Davoudi, F.~L. Lewis, D.~Naso, B.~Turchiano, Distributed
  consensus-based economic dispatch with transmission losses, IEEE Transactions
  on Power Systems 29~(4) (2014) 1711--1720.

\bibitem{Kim2020}
M.~J. Kim, T.~S. Kim, R.~J. Flores, J.~Brouwer, Neural-network-based
  optimization for economic dispatch of combined heat and power systems,
  Applied Energy 265 (2020) 114785.

\bibitem{Zhou2020}
S.~Zhou, Z.~Hu, W.~Gu, M.~Jiang, M.~Chen, Q.~Hong, C.~Booth, Combined heat and
  power system intelligent economic dispatch: A deep reinforcement learning
  approach, International Journal of Electrical Power \& Energy Systems 120
  (2020) 106016.

\bibitem{Gaing2003}
Z.-L. Gaing, Particle swarm optimization to solving the economic dispatch
  considering the generator constraints, IEEE transactions on power systems
  18~(3) (2003) 1187--1195.

\bibitem{Xin2020}
Z.~Xin-gang, L.~Ji, M.~Jin, Z.~Ying, An improved quantum particle swarm
  optimization algorithm for environmental economic dispatch, Expert Systems
  with Applications 152 (2020) 113370.

\bibitem{Kazda2020}
K.~Kazda, X.~Li, A critical review of the modeling and optimization of combined
  heat and power dispatch, Processes 8~(4) (2020) 441.

\bibitem{Wen2021}
G.~Wen, X.~Yu, Z.~Liu, Recent progress on the study of distributed economic
  dispatch in smart grid: an overview, Frontiers of Information Technology \&
  Electronic Engineering 22~(1) (2021) 25--39.

\bibitem{Saravanan2013}
B.~Saravanan, S.~Das, S.~Sikri, D.~Kothari, A solution to the unit commitment
  problem — a review, Frontiers in Energy 7~(2) (2013) 223--236.

\bibitem{Cormen2009}
T.~H. Cormen, C.~E. Leiserson, R.~L. Rivest, C.~Stein, Introduction to
  algorithms, MIT press, 2009.

\bibitem{Hindi1991}
K.~Hindi, M.~{Ab Ghani}, Dynamic economic dispatch for large scale power
  systems: a lagrangian relaxation approach, International Journal of
  Electrical Power \& Energy Systems 13~(1) (1991) 51 -- 56.

\bibitem{Ross1980}
D.~W. {Ross}, S.~{Kim}, Dynamic economic dispatch of generation, IEEE
  Transactions on Power Apparatus and Systems PAS-99~(6) (1980) 2060--2068.

\bibitem{Kanchev2014}
V.~L. H.~Kanchev, F.~Colas, B.~Francois, Emission reduction and economical
  optimization of an urban microgrid operation including dispatched pv-based
  active generators, IEEE Transactions on Sustainable Energy 5~(4) (2014)
  1397--1405.

\bibitem{Shamsi2016}
P.~{Shamsi}, H.~{Xie}, A.~{Longe}, J.~{Joo}, Economic dispatch for an
  agent-based community microgrid, IEEE Transactions on Smart Grid 7~(5) (2016)
  2317--2324.

\bibitem{Gross1987}
G.~{Gross}, F.~D. {Galiana}, Short-term load forecasting, Proceedings of the
  IEEE 75~(12) (1987) 1558--1573.

\bibitem{Akay2007}
D.~Akay, M.~Atak, Grey prediction with rolling mechanism for electricity demand
  forecasting of turkey, Energy 32~(9) (2007) 1670 -- 1675.

\bibitem{Yu2017}
C.~{Yu}, P.~{Mirowski}, T.~K. {Ho}, A sparse coding approach to household
  electricity demand forecasting in smart grids, IEEE Transactions on Smart
  Grid 8~(2) (2017) 738--748.

\bibitem{Mirasgedis2006}
S.~Mirasgedis, Y.~Sarafidis, E.~Georgopoulou, D.~Lalas, M.~Moschovits,
  F.~Karagiannis, D.~Papakonstantinou, Models for mid-term electricity demand
  forecasting incorporating weather influences, Energy 31~(2) (2006) 208 --
  227.

\bibitem{BenTal2009}
A.~Ben-Tal, L.~El~Ghaoui, A.~Nemirovski, Robust optimization, Vol.~28,
  Princeton University Press, 2009.

\bibitem{Elsayed2016}
W.~Elsayed, Y.~Hegazy, F.~Bendary, M.~El-Bages, A review on accuracy issues
  related to solving the non-convex economic dispatch problem, Electric Power
  Systems Research 141 (2016) 325--332.

\bibitem{Zhao2010}
J.~Zhao, F.~Wen, Y.~Xue, Z.~Dong, J.~Xin, Power system stochastic economic
  dispatch considering uncertain outputs from plug-in electric vehicles and
  wind generators, Dianli Xitong Zidonghua(Automation of Electric Power
  Systems) 34~(20) (2010) 22--29.

\bibitem{Hetzer2008}
J.~Hetzer, C.~Y. David, K.~Bhattarai, An economic dispatch model incorporating
  wind power, IEEE Transactions on energy conversion 23~(2) (2008) 603--611.

\bibitem{Dhillon1993}
J.~Dhillon, S.~Parti, D.~Kothari, Stochastic economic emission load dispatch,
  Electric Power Systems Research 26~(3) (1993) 179--186.

\bibitem{Bertsimas2004}
D.~Bertsimas, M.~Sim, The price of robustness, Operations research 52~(1)
  (2004) 35--53.

\bibitem{Yu1998}
G.~Yu, J.~Yang, On the robust shortest path problem, Computers \& operations
  research 25~(6) (1998) 457--468.

\bibitem{Bondy1977}
J.~Bondy, U.~Murty, Graph Theory with Applications, Macmillan, 1977.

\bibitem{Rockafellar1970}
R.~T. Rockafellar, Convex analysis, no.~28, Princeton university press, 1970.

\bibitem{Yaman2001}
H.~Yaman, O.~E. Kara\c{s}an, M.~\c{C} .~Pınar, The robust spanning tree
  problem with interval data, Operations Research Letters 29~(1) (2001) 31 --
  40.
\newblock \href {https://doi.org/https://doi.org/10.1016/S0167-6377(01)00078-5}
  {\path{doi:https://doi.org/10.1016/S0167-6377(01)00078-5}}.

\bibitem{Montemanni2004}
R.~Montemanni, L.~M. Gambardella, An exact algorithm for the robust shortest
  path problem with interval data, Computers \& Operations Research 31~(10)
  (2004) 1667--1680.

\bibitem{Bertsimas2003}
D.~Bertsimas, M.~Sim, Robust discrete optimization and network flows,
  Mathematical programming 98~(1-3) (2003) 49--71.

\bibitem{Gabrel2013}
V.~Gabrel, C.~Murat, L.~Wu, New models for the robust shortest path problem:
  complexity, resolution and generalization, Annals of Operations Research
  207~(1) (2013) 97--120.

\bibitem{Zhang2018a}
Y.~{Zhang}, S.~{Song}, Z.~M. {Shen}, C.~{Wu}, Robust shortest path problem with
  distributional uncertainty, IEEE Transactions on Intelligent Transportation
  Systems 19~(4) (2018) 1080--1090.

\bibitem{Martin2014}
A.~Martin, J.~C. M{\"u}ller, S.~Pokutta, Strict linear prices in non-convex
  european day-ahead electricity markets, Optimization Methods and Software
  29~(1) (2014) 189--221.

\bibitem{Schiro2016}
D.~A. {Schiro}, T.~{Zheng}, F.~{Zhao}, E.~{Litvinov}, Convex hull pricing in
  electricity markets: Formulation, analysis, and implementation challenges,
  IEEE Transactions on Power Systems 31~(5) (2016) 4068--4075.

\bibitem{Beaudin2015}
M.~Beaudin, H.~Zareipour, Home energy management systems: A review of modelling
  and complexity, Renewable and sustainable energy reviews 45 (2015) 318--335.

\bibitem{Majumdar1975}
S.~Majumdar, Low-cycle fatigue and creep analysis of gas turbine engine
  components, Journal of Aircraft 12~(4) (1975) 376--382.

\bibitem{EnergyBuildings1}
Commercial reference buildings,
  \url{http://energy.gov/eere/buildings/commercial-reference-buildings}
  [Accessed 2020-08-31].

\bibitem{EnergyBuildings2}
Building characteristics for residential hourly load data: based on building
  {America} house simulation protocols,
  \url{http://en.openei.org/doe-opendata/dataset/eadfbd10-67a2-4f64-a394-3176c7b686c1/resource/cd6704ba-3f53-4632-8d08-c9597842fde3/download/buildingcharacteristicsforresidentialhourlyloaddata.pdf}
  [Accessed 2020-08-31].

\bibitem{EnergyBuildings3}
Commercial and residential hourly load profiles for all {TMY3} locations in the
  united states. {Office of Energy Efficiency \& Renewable Energy (EERE)},
  \url{https://openei.org/datasets/dataset/commercial-and-residential-hourly-load-profiles-for-all-tmy3-locations-in-the-united-states}
  [Accessed 2020-08-31].

\bibitem{PSEGLINY}
Tariff for electric service residential. {PSEG} {LINY}; 2016,
  \url{https://www.psegliny.com/files.cfm/rates\_resi.pdf}.

\bibitem{PriceOfGas}
{U.S.} natural gas prices. {U.S. Energy Information Administration},
  \url{https://www.eia.gov/dnav/ng/ng\_pri\_sum\_dcu\_nus\_m.html} [Accessed
  2020-08-31].

\end{thebibliography}

\end{document}